\pdfoutput=1
\documentclass[11pt]{article}
\usepackage{authblk}
\usepackage[toc,page]{appendix}
\usepackage[top=2cm, bottom=2cm, left=2cm, right=2cm]{geometry}
\usepackage{centernot}

\usepackage{color}
\usepackage{helvet}         
\usepackage{courier}        
\usepackage{type1cm}        

\usepackage{framed}
\usepackage{tikz}
\usepackage{makeidx}         
\usepackage{graphicx}        
\usepackage{multicol}        
\usepackage[bottom]{footmisc}

\usepackage{amsmath}
\usepackage{amssymb}
\usepackage{bbold}
\usepackage{amsthm}
\usepackage{subcaption}
\usepackage{sidecap}
\usepackage{floatrow}
\usepackage{pdflscape}
\usepackage{pdflscape}
\usepackage{comment}
\usepackage[font=small]{caption}
\usepackage{enumitem}
\usepackage{esint}

\usepackage{scalerel}

\newtheorem{theorem}{Theorem}
\newtheorem{corollary}[theorem]{Corollary}
\newtheorem{lemma}[theorem]{Lemma}

\newtheorem{proposition}[theorem]{Proposition}

\theoremstyle{remark}

\newtheorem{remark}[theorem]{\bf Remark}

\numberwithin{theorem}{section}
\numberwithin{question}{section}
\numberwithin{figure}{section}
\numberwithin{equation}{section}

\allowdisplaybreaks[4]

\begin{document}

\title{Conformal covariance of connection probabilities \\ in the 2D critical FK-Ising model}
\bigskip{}
\author[1]{Federico Camia}
\author[2]{Yu Feng}

\affil[1]{NYU Abu Dhabi, UAE \& Courant Institute, USA}
\affil[2]{Tsinghua University, China}

\date{}


\global\long\def\CR{\mathrm{CR}}
\global\long\def\ST{\mathrm{ST}}
\global\long\def\SF{\mathrm{SF}}
\global\long\def\cov{\mathrm{cov}}
\global\long\def\dist{\mathrm{dist}}
\global\long\def\SLE{\mathrm{SLE}}
\global\long\def\hSLE{\mathrm{hSLE}}
\global\long\def\CLE{\mathrm{CLE}}
\global\long\def\GFF{\mathrm{GFF}}
\global\long\def\inte{\mathrm{int}}
\global\long\def\ext{\mathrm{ext}}
\global\long\def\inrad{\mathrm{inrad}}
\global\long\def\outrad{\mathrm{outrad}}
\global\long\def\dimH{\mathrm{dim}}
\global\long\def\capa{\mathrm{cap}}
\global\long\def\diam{\mathrm{diam}}
\global\long\def\free{\mathrm{free}}
\global\long\def\Dist{\mathrm{Dist}}
\global\long\def\hF{{}_2\mathrm{F}_1}
\global\long\def\simple{\mathrm{simple}}
\global\long\def\even{\mathrm{even}}
\global\long\def\odd{\mathrm{odd}}
\global\long\def\st{\mathrm{ST}}
\global\long\def\usf{\mathrm{USF}}
\global\long\def\Leb{\mathrm{Leb}}
\global\long\def\LP{\mathrm{LP}}
\global\long\def\coulomb{\LH}
\global\long\def\coulombnew{\LG}
\global\long\def\kfunc{p}
\global\long\def\OO{\mathcal{O}}
\global\long\def\parti{\mathbf{Q}}
\global\long\def\rad{\mathrm{rad}}
\global\long\def\IsingHfree{\mathbb{E}_{\mathbb{H}}^{(a,\mathrm{f})}}
\global\long\def\oH{\overline{\mathbb{H}}}
\global\long\def\aIsingHplus{\widehat{\mathbb{E}}_{\mathbb{H}}^{\dot{\boldsymbol{x}}^a}}

\global\long\def\eps{\epsilon}
\global\long\def\ov{\overline}
\global\long\def\U{\mathbb{U}}
\global\long\def\T{\mathbb{T}}
\global\long\def\HH{\mathbb{H}}
\global\long\def\LA{\mathcal{A}}
\global\long\def\LB{\mathcal{B}}
\global\long\def\LC{\mathcal{C}}
\global\long\def\LD{\mathcal{D}}
\global\long\def\LF{\mathcal{F}}
\global\long\def\LK{\mathcal{K}}
\global\long\def\LE{\mathcal{E}}
\global\long\def\LG{\mathcal{G}}
\global\long\def\LI{\mathcal{I}}
\global\long\def\LJ{\mathcal{J}}
\global\long\def\LL{\mathcal{L}}
\global\long\def\LM{\mathcal{M}}
\global\long\def\LN{\mathcal{N}}
\global\long\def\LQ{\mathcal{Q}}
\global\long\def\LR{\mathcal{R}}
\global\long\def\LT{\mathcal{T}}
\global\long\def\LS{\mathcal{S}}
\global\long\def\LU{\mathcal{U}}
\global\long\def\LV{\mathcal{V}}
\global\long\def\LW{\mathcal{W}}
\global\long\def\LX{\mathcal{X}}
\global\long\def\LY{\mathcal{Y}}
\global\long\def\PartF{\mathcal{Z}}
\global\long\def\LH{\mathcal{H}}
\global\long\def\LJ{\mathcal{J}}
\global\long\def\R{\mathbb{R}}
\global\long\def\C{\mathbb{C}}
\global\long\def\N{\mathbb{N}}
\global\long\def\Z{\mathbb{Z}}
\global\long\def\E{\mathbb{E}}
\global\long\def\PP{\mathbb{P}}
\global\long\def\QQ{\mathbb{Q}}
\global\long\def\A{\mathbb{A}}
\global\long\def\one{\mathbb{1}}
\global\long\def\bn{\mathbf{n}}
\global\long\def\MR{MR}
\global\long\def\cond{\,|\,}
\global\long\def\la{\langle}
\global\long\def\ra{\rangle}
\global\long\def\tree{\Upsilon}
\global\long\def\prob{\mathbb{P}}
\global\long\def\hm{\mathrm{Hm}}
\global\long\def\cross{\mathrm{Cross}}

\global\long\def\sf{\mathrm{SF}}
\global\long\def\wr{\varrho}

\global\long\def\Im{\operatorname{Im}}
\global\long\def\Re{\operatorname{Re}}

\global\long\def\ud{\mathrm{d}}
\global\long\def\pder#1{\frac{\partial}{\partial#1}}
\global\long\def\pdder#1{\frac{\partial^{2}}{\partial#1^{2}}}
\global\long\def\der#1{\frac{\ud}{\ud#1}}

\global\long\def\bZnn{\mathbb{Z}_{\geq 0}}

\global\long\def\Vfunc{\LG}
\global\long\def\gfunc{g^{(\rr)}}
\global\long\def\hfunc{h^{(\rr)}}

\global\long\def\SimplexInt{\rho}
\global\long\def\CubeInt{\widetilde{\rho}}

\global\long\def\ii{\mathfrak{i}}
\global\long\def\rr{\mathfrak{r}}
\global\long\def\chamber{\mathfrak{X}}
\global\long\def\Wchamber{\mathfrak{W}}

\global\long\def\SimplexIntKappa8{\SimplexInt}

\global\long\def\nested{\boldsymbol{\underline{\Cap}}}
\global\long\def\unnested{\boldsymbol{\underline{\cap\cap}}}
\global\long\def\unnested{\boldsymbol{\underline{\cap\cap}}}

\global\long\def\acycle{\vartheta}
\global\long\def\bcycle{\tilde{\acycle}}
\global\long\def\Gloop{\Theta}

\global\long\def\metric{\mathrm{dist}}

\global\long\def\adj#1{\mathrm{adj}(#1)}

\global\long\def\bs{\boldsymbol}

\global\long\def\edge#1#2{\langle #1,#2 \rangle}
\global\long\def\graph{G}

\newcommand{\conn}{\vartheta_{\scaleobj{0.7}{\mathrm{RCM}}}}
\newcommand{\hatconn}{\widehat{\vartheta}_{\mathrm{RCM}}}
\newcommand{\realpt}{\smash{\mathring{x}}}
\newcommand{\corrind}{\LC}
\newcommand{\bssymb}{\pi}
\newcommand{\PRCM}{\mu}
\newcommand{\coeff}{p}
\newcommand{\MainConst}{C}

\global\long\def\removeLink{/}

\maketitle

\begin{abstract}
    We study connection probabilities between vertices of the square lattice for the critical random-cluster (FK) model with cluster weight 2, which is related to the critical Ising model. We consider the model on the plane and on domains conformally equivalent to the upper half-plane. We prove that, when appropriately rescaled, the connection probabilities between vertices in the domain or on the boundary have nontrivial limits, as the mesh size of the square lattice is sent to zero, and that those limits are conformally covariant. This provides an important step in the proof of the Delfino-Viti conjecture for FK-Ising percolation as well as an alternative proof of the conformal covariance of the Ising spin correlation functions. In an appendix, we also derive new exact formulas for some Ising boundary spin correlation functions.
\end{abstract}



\tableofcontents

\section{Introduction}
\subsection{Background and motivation}\label{sec::bac::mot}
Fortuin and Kasteleyn introduced the \emph{random-cluster model} in the 1970s (see \cite{FORTUIN1972536}) as a general family of discrete percolation models that combines together Bernoulli percolation, graphical representations of spin models (Ising \& Potts models), and polymer models (as a limiting case).
Generally, in such models, edges are declared open or closed according to a given probability measure, the simplest being the independent product measure of Bernoulli percolation.
Of particular interest are percolation properties, that is, whether various points in space are connected by paths of open edges. 

The random-cluster model has been actively investigated in the past decades, for instance, because of its important feature of \emph{criticality}: for certain parameter values the model exhibits a continuous phase transition.
Criticality can be practically identified as follows.
On a lattice with a small mesh, say $\delta \Z^2$, consider the probability that an open path connects two opposite sides of a topological rectangle (i.e.,~a bounded domain with four marked points on its boundary). This probability tends to zero as $\delta \to 0$ when the model is ``subcritical,''  
while it tends to one as $\delta \to 0$ when the model is ``supercritical.'' 
At the critical point,  
the connection probability has a nontrivial limit, which belongs to $(0,1)$ and depends on the ``shape'' (i.e.,~the conformal modulus) of the topological rectangle.
The exact identification of the limit of the connection probability, though, is highly nontrivial.

The phase transition in the random-cluster model has been argued to result in conformal invariance and universality for the scaling limit of the model (see, e.g.,~\cite{Cardy:Scaling_and_renormalization_in_statistical_physics}). 
For generic values of the cluster weight parameter $q \in [1,4]$, it was recently shown~\cite{DKKMO:Rotational_invariance_in_critical_planar_lattice_models} that correlations in the critical random-cluster model become rotationally invariant in the scaling limit. This provides strong evidence of conformal invariance, while still not being enough to prove it. 
Conformal invariance had been previously rigorously established for the FK-Ising model (cluster weight $q=2$) and for Bernoulli site percolation on the triangular lattice (related to Bernoulli bond percolation, corresponding to cluster weight $q=1$) \cite{SmirnovPercolationConformalInvariance,CamiaNewmanPercolationFull,CamiaNewmanPercolation,SmirnovHolomorphicFermion,ChelkakSmirnovIsing,CDHKS:Convergence_of_Ising_interfaces_to_SLE,KemppainenSmirnovFullLimitFKIsing,KemppainenSmirnovBoundaryTouchingLoopsFKIsing,IzyurovMultipleFKIsing}. 

In addition to proving conformal invariance, identifying in the scaling limit objects that have a conformal field theory (CFT) interpretation is crucial in order to get access to the full power of the CFT formalism applicable to critical lattice models (see, e.g.,~\cite{henkel2013conformal}). In this direction, in the case of critical site percolation on the triangular lattice, one of us recently established~\cite{Cam24,PercGasket} the conformal covariance of connection probabilities in the scaling limit, showing that they can be interpreted as CFT correlation functions and proving a conjecture formalized by Aizenman in the 1990s. 
We then moved one step forward and started to explore the CFT structure of critical percolation~\cite{CamiaFengLogPercolationMath,CamiaFengLogPercolationPhysics}, identifying the scaling limits of various connection probabilities with CFT correlation functions and proving a rigorous version of an operator product expansion (OPE).

The first main motivation of this article is to provide 
a natural extension of the aforementioned works~\cite{Cam24,PercGasket} to the FK-Ising model, which is of great interest to both mathematicians and physicists. 
In those works, the local independence of percolation is used in the proofs, so it is natural to ask whether one can adapt the arguments developed for percolation to deal with the critical random-cluster model with cluster weight $q\neq1$.
In this paper, we focus on the case $q=2$, the only one for which the conformal invariance of the scaling limit of interfaces has been proved so far.
As we will see, extending the results of~\cite{Cam24,PercGasket} to the FK model with $q=2$ requires additional work and involves new ingredients, namely a classical result by Wu~\cite{WUTwopoint} on Ising two-point functions, a spatial mixing property and, in the case of connection probabilities involving boundary points, Smirnov's FK-Ising fermionic observable (see~\cite{SmirnovHolomorphicFermion})\footnote{Wu’s result on the Ising two-point function and Smirnov’s FK-Ising observable are
only used to figure out the exact orders of the normalization factors in Theorems 1.1 and 1.4.}.

The second main motivation is to provide an alternative approach to study the conformal covariance and the CFT structure of spin and energy correlations in the Ising and Potts models, which are classical models of ferromagnetism and are among the most studied models of statistical mechanics.
In the case of the Ising model, the conformal covariance and the CFT structure of spin and energy correlations have been established rigorously to a large extent~\cite{HonglerSmirnovIsingEnergy,ChelkakIzyHolomorphic,ChelkakHonglerIzyurovConformalInvarianceCorrelationIsing,CHI:Correlations_of_primary_fields_in_the_critical_planar_Ising_model,UniversalitySpinIsing} using discrete complex analysis tools, where the \emph{s-holomorphicity} of certain observables plays an essential role. However, s-holomorphicity is difficult to prove beyond the cases of the Ising and FK-Ising models.
Since the correlations of some of the most basic Ising and Potts fields, such as the spin and energy fields, can be expressed in terms of point-to-point connection probabilities in the random-cluster model via the Edwards-Sokal coupling\footnote{In particular, the FK-Ising random-cluster model is related to the Ising spin model.}~\cite{EdwardsSokal}, it is interesting to develop a geometric approach to study conformal covariance and the CFT structure of spin and energy correlations based on connection probabilities and interfaces in the random-cluster model.\footnote{It would also be interesting to construct spin or energy correlations directly for $\mathrm{CLE}$ in the continuum.}
Such an approach is already interesting for the case of the Ising model, but could prove potentially even more useful to study the scaling limits of Potts model with values of $q\neq 2$.

We will show that, for the 2D critical FK-Ising model, (normalized) point-to-point connection probabilities of various kinds of link patterns have conformally covariant scaling limits (see Theorems~\ref{thm::cvg_proba} and~\ref{thm::main_boundary} below).
As a corollary, we provide a new proof of conformal covariance of Ising spin correlations (see Corollary~\ref{coro::two_point_FK} below).
The main inputs of the proofs are the FKG inequality, RSW estimates, the one-arm exponent for $\mathrm{CLE}$ computed in~\cite{SchrammSheffieldWilsonConformalRadii}, and the convergence of interfaces towards $\mathrm{CLE}_{16/3}$ in the Camia-Newman topology\footnote{See~\cite{CamiaNewmanPercolationFull}.}~\cite{KemppainenSmirnovFullLimitFKIsing,KemppainenSmirnovBoundaryTouchingLoopsFKIsing}.
 We also use a spatial mixing property, which is essentially a consequence of the FKG inequality and RSW estimates, as shown in~\cite{DuminilCopinHonglerNolinRSWFKIsing}. We note that, although \cite{DuminilCopinHonglerNolinRSWFKIsing} deals with FK percolation with $q=2$, there seems to be no fundamental obstacle to extending the arguments in that paper to other values of $q \in [1,4]$, given the corresponding RSW estimates established in~\cite{DCMTRCMFractalProperties}.

In the present paper, results proved using discrete complex analysis techniques are needed directly only when dealing with correlation functions involving boundary vertices, namely in Section~\ref{sec::con::bou}\footnote{In Section~\ref{sec::con::bou}, they are used to replace a classical result by Wu on Ising correlations between pairs of points in the bulk.} and in the appendix. They are also used indirectly because the proofs of convergence of discrete interfaces towards $\mathrm{CLE}_{16/3}$ involve the s-holomorphicity of certain observables (see~\cite{KemppainenSmirnovFullLimitFKIsing,KemppainenSmirnovBoundaryTouchingLoopsFKIsing}). However, the recent groundbreaking work~\cite{DKKMO:Rotational_invariance_in_critical_planar_lattice_models} suggests that a proof of convergence and conformal invariance of interfaces for $q\in [1,4]$ without using s-holomorphicity may be possible in the future.

We emphasize that, for our results involving only vertices in the bulk, the convergence of interfaces to $\mathrm{CLE}_{16/3}$ is the only place where s-holomorphicity is used.
If one could prove convergence to $\mathrm{CLE}_{\kappa}$ for other FK models with $q \in [1,4]$, then a combination of our arguments in this paper and standard percolation techniques would allow us to extend our results to those FK models, at least in a weaker form (normalizing connection probabilities with the probability of the one-arm event).

\subsection{Random-cluster model}

For definiteness and to take full advantage of known scaling limit results (see~\cite{SmirnovHolomorphicFermion,CDHKS:Convergence_of_Ising_interfaces_to_SLE}), we consider subgraphs of the square lattice $\Z^2$, which is the graph with vertex set $V(\Z^2):=\{ z = (m, n) \colon m, n\in \Z\}$ and edge set $E(\Z^2)$ given by edges $\edge{z}{w}$ between vertices $z,w \in V(\mathbb{Z}^2)$ whose Euclidean distance equals one (called \emph{neighbors}).
This is our primal lattice. Its standard dual lattice is denoted by $(\Z^2)^{\bullet}$. 
The medial lattice $(\Z^2)^{\diamond}$ is the graph {whose vertices are the centers of the edges of the square lattice and whose edges connect vertices at distance $1/\sqrt{2}$.} 
For a subgraph $\graph \subset \Z^2$, we define its boundary to be the following set of vertices:
\begin{align*}
	\partial \graph = \{ z \in V(\graph) \, \colon \, \exists \; w \not\in V(\graph) \text{ such that }\edge{z}{w}\in E(\Z^2)\} ,
\end{align*}
{and similarly for subgraphs of $(\Z^2)^{\bullet}$ and $(\Z^2)^{\diamond}$.}
When we add the subscript or superscript $a$, we mean that the lattices $\Z^2, (\Z^2)^{\bullet}, (\Z^2)^\diamond$ have been scaled by $a > 0$. 
We consider the models in the scaling limit $a \to 0$. For $z\in \mathbb{C}$ and $r>0$, we write
\begin{equation*}
	B_r(z)=\{w\in \mathbb{C}:|z-w|<r\}.
\end{equation*}

Let $\graph = (V(\graph), E(\graph))$ be a finite subgraph of $\Z^2$.
A random-cluster \textit{configuration} 
$\omega=(\omega_e)_{e \in E(\graph)}$ is an element of $\{0,1\}^{E(\graph)}$.
An edge $e \in E(\graph)$ is said to be \textit{open} (resp.~\textit{closed}) if $\omega_e=1$ (resp.~$\omega_e=0$).
We view the configuration $\omega$ 
as a subgraph of $\graph$ with vertex set $V(\graph)$  and edge set $\{e\in E(\graph) \colon \omega_e=1\}$.
We denote by $o(\omega)$ (resp.~$c(\omega)$) the number of open (resp.~closed) edges in~$\omega$.

We are interested in the connectivity properties of the graph $\omega$ with various boundary conditions. 
The maximal connected\footnote{Two vertices $z$ and $w$ are said to be \textit{connected} by $\omega$ if there exists a sequence  $\{z_j \colon 0\le j\le l\}$  of vertices such that
	$z_0 = z$, $z_l = w$, and each edge $\edge{z_j}{z_{j+1}}$ is open in $\omega$ for $0 \le j < l$.} components of $\omega$ are called \textit{clusters}.
The boundary conditions encode how the vertices are connected outside of $\graph$.
More precisely, by a \textit{boundary condition} $\bssymb$ we refer to a partition $\bssymb_1 \sqcup \cdots \sqcup \bssymb_m$ of $\partial \graph$.
Two vertices $z,w \in \partial \graph$ are said to be \textit{wired} in $\bssymb$ if $z,w \in \bssymb_j$ for some $j$. 
In contrast, \textit{free} boundary segments comprise vertices that are not wired with any other vertex (so the corresponding part $\pi_j$ is a singleton).  
We denote by $\omega^{\bssymb}$
the (quotient) graph obtained from the configuration $\omega$ by identifying the wired vertices in $\bssymb$.

Finally, the \textit{random-cluster model} on $\graph$ with edge-weight $p\in [0,1]$, cluster-weight $q>0$,
and boundary condition $\bssymb$, is the probability measure $\smash{\PRCM^{\bssymb}_{p,q,\graph}}$ on
the set $\{0,1\}^{E(\graph)}$ of configurations $\omega$  defined by
\begin{align*}
	\PRCM^{\bssymb}_{p,q,\graph}[\omega] 
	:= \; & \frac{p^{o(\omega)}(1-p)^{c(\omega)}q^{k(\omega^{\bssymb})}}{\underset{\omega \in \{0,1\}^{E(\graph)}}{\sum} p^{o(\omega)}(1-p)^{c(\omega)}q^{k(\omega^{\bssymb})} } ,
\end{align*}
where $k(\omega^{\bssymb})$ is the number of 
connected components of the graph $\omega^{\bssymb}$.
For $q=2$, this model is also known as the \textit{FK-Ising model}, while for $q=1$, it is simply the Bernoulli bond percolation (i.e., it is a product measure, with the edges taking independent values). The random-cluster model combines together several important models in the same family.
For integer values of $q$, it is very closely related to the $q$-state Potts model, 
and by taking a suitable limit, the case of $q=0$ corresponds to the uniform spanning tree (see, e.g.,~\cite{Duminil-Copin-PIMS_lectures}). 
For $q \in [1,4]$, it was proven~\cite{DuminilSidoraviciusTassionContinuityPhaseTransition} that, for a suitable choice of edge-weight $p$, namely
\begin{align*}
	p = p_c(q) := \frac{\sqrt{q}}{1+\sqrt{q}} ,
\end{align*}
the random-cluster model exhibits a \textit{continuous phase transition}
in the sense that, for $p > p_c(q)$, there almost surely exists an infinite cluster, while for $p < p_c(q)$, there is no infinite cluster almost surely.
Moreover, the limit $p \searrow p_c(q)$ is approached in a continuous way. 
(This is also expected to hold when $q \in (0,1)$, while it is known that the phase transition is discontinuous when $q > 4$~\cite{DCGagnebinHarelManolescuTassionDiscontinuity}.)
Therefore, the scaling limit is expected to be conformally invariant for all $q \in [0,4]$.
In the present article, we consider point-to-point connection probabilities in the critical FK-Ising model.

\subsection{Connection probabilities of interior vertices}
Fix $n\geq 2$, let $\parti=(Q_1,\ldots,Q_r)$ be a partition of $\{1,2,\ldots,n\}$.  For $a>0$, we denote by $a\mathbb{Z}^2$ the scaled square lattice. For a simply connected subgraph $\Omega^{a}\subseteq a\mathbb{Z}^2$, we denote by $\mathbb{P}^{a}=\mathbb{P}^a_{\Omega}$ the critical FK-Ising measure on $\Omega^{a}$ with free boundary conditions\footnote{The boundary condition chosen here is not essential. We can change to, for instance, wired or alternating wired/free boundary conditions.}. Let $z_1^a,\ldots,z_n^a\in \Omega^{a}$ be $n$ distinct vertices. Denote by $G(\parti;z_1^a,\ldots,z_n^a)$ the event that $z_1^a,\ldots,z_n^a$ are connected to each other according to the partition $\parti$, meaning that $z_i^a$ and $z_j^a$ are in the same open cluster if and only if $i$ and $j$ are in the same element of $\bs{Q}$. See Figure~\ref{fig:link-pattern-four} for a schematic example.
\begin{figure}
	\begin{subfigure}[b]{0.23\textwidth}
		\begin{center}
			\includegraphics[width=1\textwidth]{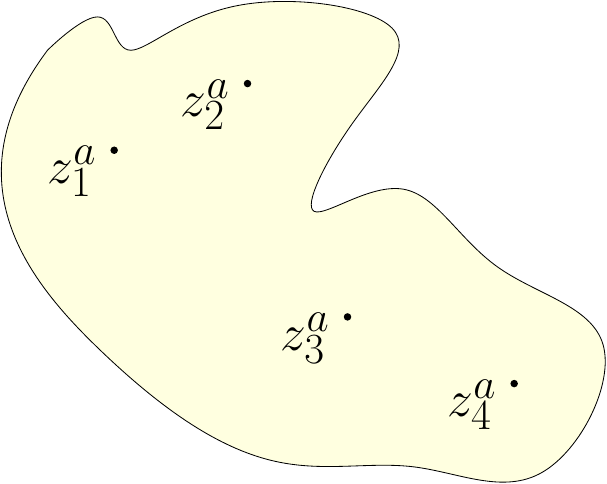}
		\end{center}
		\caption{}
	\end{subfigure}
	\begin{subfigure}[b]{0.23\textwidth}
		\begin{center}
			\includegraphics[width=1\textwidth]{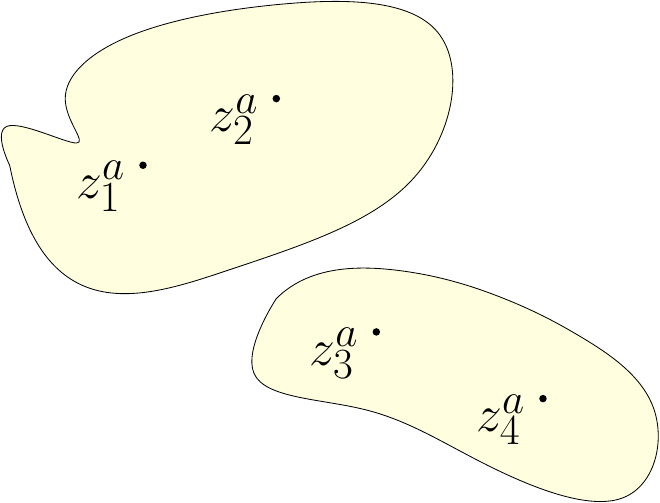}
		\end{center}
		\caption{}
	\end{subfigure}
	\begin{subfigure}[b]{0.23\textwidth}
		\begin{center}
			\includegraphics[width=1\textwidth]{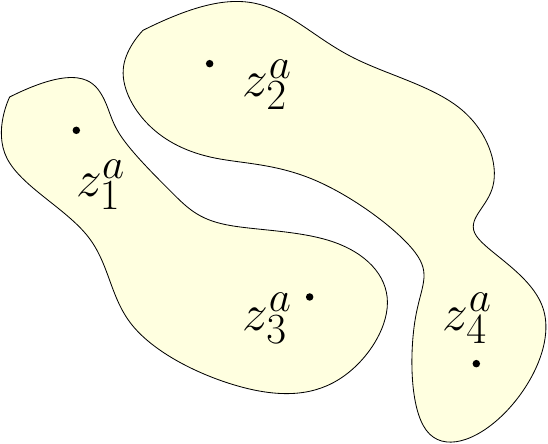}
		\end{center}
		\caption{}
	\end{subfigure}
	\begin{subfigure}[b]{0.23\textwidth}
		\begin{center}
			\includegraphics[width=1\textwidth]{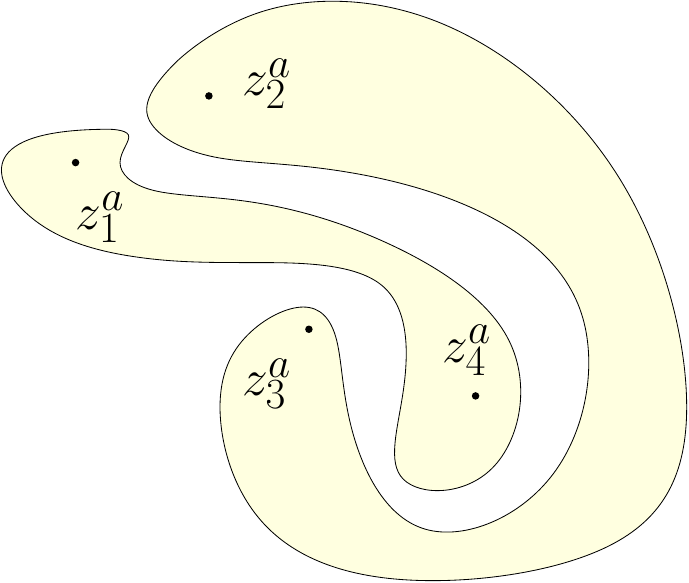}
		\end{center}
		\caption{}
	\end{subfigure}
    \caption{The event $G(\parti;z_1^a,z_2^a,z_3^a,z_4^a)$ with (a) $\parti= (\{1,2,3,4\})$; (b) $\parti = (\{1,2\},\{3,4\})$; (c) $\parti= (\{1,3\},\{2,4\})$; and (d) $\parti = (\{1,4\},\{2,3\})$. The yellow regions represent open clusters and two yellow regions are disjoint if and only if they represent distinct open clusters.}
       \label{fig:link-pattern-four}
\end{figure}
\begin{theorem} \label{thm::cvg_proba}
	Let $\Omega\subseteq \mathbb{C}$ be a simply connected domain and $z_1,\ldots,z_n\in \Omega$ be $n$ distinct points. Let $\Omega^{a}\subseteq a\mathbb{Z}^2$ be a sequence of simply connected domains that converges to $\Omega$ in the  Hausdorff metric\footnote{Our proofs can be generalized to the Carath\'edory convergence of discrete domains easily, which is weaker than the Hausdorff convergence used in this paper for simplicity.} as $a\to 0$. Suppose that $z_1^a,\ldots,z_{n}^a\in \Omega^a$ are vertices satisfying $\lim_{a\to 0}z_j^a = z_j$ for $1\leq j\leq n$. Let $\parti$ be a partition of $\{1,2,\ldots,n\}$ that contains no singletons. Then we have the following:
\begin{enumerate}[label=\textnormal{(\arabic*)}, ref=(\arabic*)]
    \item The limit 
    \begin{equation} \label{eqn::def_P}
		P(\Omega;\parti;z_1,\ldots,z_n):=\lim_{a\to 0}a^{-\frac{n}{8}}\times \mathbb{P}^a_{\Omega}\left[G(\parti;z_1^a,\ldots,z_{n}^a)\right]
	\end{equation}
exists and belongs to $(0,\infty)$.  \label{item::thm_cvg_proba}
\item The function $P$ defined via~\eqref{eqn::def_P} satisfies the following conformal covariance property: if $\varphi$ is a conformal map from $\Omega$ onto some $\Omega'$ such that $\varphi(z_j)\neq \infty$ for $1\leq j\leq n$, then we have 
	\begin{equation} \label{eqn::conformal_covariance}
P(\Omega';\parti;\varphi(z_1),\ldots,\varphi(z_n))=P(\Omega;\parti;z_1,\ldots,z_n)\times \prod_{j=1}^n |\varphi'(z_j)|^{-\frac{1}{8}}.
	\end{equation} \label{item::thm::cov}
\end{enumerate}
\end{theorem}

The normalization factor $a^{\frac{1}{8}}$ in~\eqref{thm::cvg_proba} is related to the interior one-arm exponent for the FK-Ising model and can be derived using Wu's result on the full-plane Ising two-point spin correlation (see~\cite{WUTwopoint} and Theorem~\ref{thm::two_point_Ising} below for more details). As we will see in the proof, without Wu's result (Theorem~\ref{thm::two_point_Ising}), we can still prove the results in Theorem~\ref{thm::cvg_proba} by combining Lemmas~\ref{lem::cvg_proba_aux1},~\ref{lem::norma_1} and~\ref{lem::scaling} below, but with the normalization factor $a^{\frac{1}{8}}$ replaced by $\mathbb{P}_{\mathbb{Z}^2}^a[0\longleftrightarrow \partial B_1(0)]$.



We emphasize that the domain $\Omega$ in Theorem~\ref{thm::cvg_proba} is not necessarily bounded. For $n\geq 2$, we write $\parti_{n}=(\{1,2,\ldots,n\})$ for the special partition with a single element, corresponding to the case in which $z_1^a,\ldots z_n^a$ belong to the same cluster. Then Theorem~\ref{thm::cvg_proba} immediately implies that there exists a constant $C_1\in (0,\infty)$ such that
\begin{equation*}
	P(\mathbb{C};\parti_2;z_1,z_2)=C_1|z_1-z_2|^{-\frac{1}{4}},
\end{equation*}
which can also be derived from the rotational invariance of the full-plane Ising correlations given by~\cite[Remark~2.26]{ChelkakHonglerIzyurovConformalInvarianceCorrelationIsing} (or~\cite{RotationInvari}) and the Edwards-Sokal coupling (see~\cite{EdwardsSokal}). 
Moreover, since M\"obius transformations have three degrees of freedom, we can also conclude from Theorem~\ref{thm::cvg_proba} that there exists a constant $C_2\in (0,\infty)$ such that 
\begin{equation} \label{eqn::P3}
	P(\mathbb{C};\parti_3;z_1,z_2,z_3)=C_2|z_1-z_2|^{-\frac{1}{8}}|z_1-z_3|^{-\frac{1}{8}}|z_2-z_3|^{-\frac{1}{8}}.
\end{equation}
Consequently, we have the following factorization formula
\begin{equation*}
	P(\mathbb{C};\parti_3;z_1,z_2,z_3)=\frac{C_2}{C_1^{\frac{3}{2}}}\sqrt{	P(\mathbb{C};\parti_2;z_1,z_2)\times 	P(\mathbb{C};\parti_2;z_1,z_3)\times 	P(\mathbb{C};\parti_2;z_2,z_3)}.
\end{equation*}
Analogous results are derived in~\cite[Section~1.1]{Cam24} for percolation, in which case, the value of the constant $C_2/C_1^{3/2}$, first conjectured by Delfino and Viti in~\cite{Delfino_2011}, was recently computed rigorously in~\cite{IntegraCLE}, using techniques that rely on Liouville quantum gravity and the imaginary DOZZ formula~\cite{Schomerus_2003,Zam05,KP07}. According to private communications with the authors of~\cite{IntegraCLE}, similar techniques should also allow to compute the ratio $C_2/C_1^{3/2}$ for the FK-Ising model. Combined with the results of this paper, such a computation would provide a complete proof of the Delfino-Viti conjecture for the FK-Ising model.

As another application of Theorem~\ref{thm::cvg_proba}, we can partially recover results from~\cite{ChelkakHonglerIzyurovConformalInvarianceCorrelationIsing}:

\begin{corollary}
	Assume the same setup as in Theorem~\ref{thm::cvg_proba}. Consider the critical Ising model on $\Omega^a$ with free\footnote{The boundary condition here is not essential. We can change to, for instance, $\oplus$ boundary condition or alternating $\oplus/{\mathrm{free}}$ boundary conditions.} boundary condition and denote by $\mathbb{E}^a_{\Omega}$ the corresponding expectation. Then 
	\begin{equation*}
		\langle \sigma_{z_1}\ldots\sigma_{z_n}\rangle_{\Omega}:=\lim_{a\to 0}a^{-\frac{n}{8}} \times \mathbb{E}^a_{\Omega}\left[\sigma_{z_1^a}\ldots\sigma_{z_n^a}\right]
	\end{equation*}
	exists and belongs to $[0,\infty)$. The limit equals $0$ if and only if $n$ is odd. Moreover, $\langle \sigma_{z_1}\cdots\sigma_{z_n}\rangle_{\Omega}$ satisfies the same conformal covariance property as in~\eqref{eqn::conformal_covariance}.
\end{corollary}
\begin{proof}
	Let $\mathcal{Q}$ be the set of all partitions $\parti=\left(Q_1,\ldots,Q_l\right)$ of $\{1,2,\ldots,n\}$ such that each $Q_r$ contains an even number of elements. According to the Edwards-Sokal coupling (see~\cite{EdwardsSokal}), we have
	\begin{equation}\label{eqn::Ising_npoint_aux1}
		\mathbb{E}^a_{\Omega}\left[\sigma_{z_1^a}\cdots\sigma_{z_n^a}\right]=\sum_{\parti\in \mathcal{Q}} \mathbb{P}^a_{\Omega}\left[G\left(\parti;z_1^a,\ldots,z_{n}^a\right)\right].
	\end{equation}
	Then the desired conclusions follow immediately from~\eqref{eqn::Ising_npoint_aux1} and Theorem~\ref{thm::cvg_proba}.
\end{proof}

The proof of Theorem~\ref{thm::cvg_proba} follows the spirit in~\cite{Cam24}, that is, relating connection probabilities of interior vertices to the probabilities of events involving interfaces on the lattice and CLE loops in the continuum, conditional on certain crossing events that have probability $0$ in the continuum. However, compared with the percolation case in~\cite{Cam24}, in the present case, one encounters additional difficulties due to the lack of independence of the states (open or closed) of different edges. We deal with this problem using the so-called spatial mixing property proved in~\cite{DuminilCopinHonglerNolinRSWFKIsing}, which intuitively reads as follows: given two events $\mathcal{A}_1$ and $\mathcal{A}_2$ that depend only on the states of edges inside edge sets $E_1$ and $E_2$, respectively, then $\mathcal{A}_1$ and $\mathcal{A}_2$ are almost independent when $E_1$ is far from $E_2$ (see Lemma~\ref{lem::spatial} for more details).

We denote by $\overline{\mathbb{P}}^a=\overline{\mathbb{P}}^a_{\Omega}$ the critical FK-Ising measure on $\Omega^a$ with wired boundary conditions. The same strategy can be used to show the following result {(an analogous result for percolation is proved in~\cite{PercGasket})}:
\begin{theorem} \label{thm::onepoint_boundary}
	Let $\Omega\subset \mathbb{C}$ be a simply connected domain and $z\in \Omega$. Let $\Omega^{a}\subseteq a\mathbb{Z}^2$ be a sequence of simply connected domains that converges to $\Omega$ under the Hausdorff metric as $a\to 0$. Suppose that $z^a\in \Omega$ satisfies $\lim_{a\to 0} z^a=z$. Then, there exists a constant $C_3\in (0,\infty)$ such that 
	\begin{equation} \label{eqn::onepoint_boundary}
		g(\Omega;z):=\lim_{a\to 0}a^{-\frac{1}{8}}\times \overline{\mathbb{P}}_{\Omega}^a\left[z^a\longleftrightarrow \partial\Omega^a\right]=C_3 \rad(z,\Omega)^{-\frac{1}{8}},
	\end{equation}
where $\rad(z,\Omega)$ denotes the conformal radius of $\Omega$ from $z$. 
\end{theorem}
We denote by $\mathbb{E}_{\Omega}^{(a,\mathbf{p})}$ the expectation of the critical Ising measure on $\Omega^a$ with $\oplus$ boundary condition. Thanks to the Edwards-Sokal coupling (see~\cite{EdwardsSokal}), we have 
\begin{equation} \label{eqn::Ising_mag}
\mathbb{E}_{\Omega}^{(a,\mathbf{p})}\left[\sigma_{z^a}\right]=\overline{\mathbb{P}}_{\Omega}^a\left[z^a\longleftrightarrow\partial\Omega^a\right].
\end{equation} 
Consequently, a combination of Theorem~\ref{thm::onepoint_boundary} and~\eqref{eqn::Ising_mag} gives the scaling limit of the Ising magnetization $\mathbb{E}_{\Omega}^{(a,\mathbf{p})}\left[\sigma_{z^a}\right]$ normalized by $a^{-\frac{1}{8}}$, which was derived in~\cite[Corollary~1.3]{ChelkakHonglerIzyurovConformalInvarianceCorrelationIsing} using discrete complex analysis techniques, with an explicit constant $C_3=2^{\frac{5}{12}}e^{-\frac{3}{2}\zeta'(-1)}$, where $\zeta'$ denotes the derivative of Riemann's zeta function.

\subsection{Connection probabilities involving boundary vertices}
{This section concerns the case in which some (or all) of $z_1^a,\ldots,z_{n}^a$ are on the boundary of $\Omega^a$.
In such a situation, we can derive results similar to those in the previous section, but with different normalization factors for the points on the boundary.} 

For simplicity, we only consider the critical FK-Ising model on the scaled upper half-plane $a\left(\mathbb{H}\cap \mathbb{Z}^2\right)$ with free boundary condition on $\mathbb{R}$. 
We use $\mathbb{P}_{\mathbb{H}}^a$ to denote the corresponding measure. 
\begin{theorem} \label{thm::main_boundary}
		Let $n,\ell$ be non-negative integers such that $n+\ell\geq 2$. Let $z_1,\ldots,z_n\in \mathbb{H}$ and $x_1,\ldots,x_{\ell}\in \mathbb{R}$.
		Suppose that $z_1^a,\ldots,z_{n}^a\in a(\mathbb{H}\cap \mathbb{Z}^2)$ and $x_1^a,\ldots,x_{\ell}^a\in \partial(\mathbb{H}\cap \mathbb{Z}^2)$ are  vertices satisfying $\lim_{a\to 0}z_j^a=z_j, \lim_{a\to 0}x_k^a=x_k$ for $1\leq j\leq n$ and $1\leq k\leq \ell$. Let $\parti$ be a partition of $\{1,2,\ldots,n+\ell\}$ that contains no singletons. Then 
		\begin{equation} \label{eqn::def::R}
			R(\parti;z_1,\ldots,z_n;x_1,\ldots,x_{\ell}):=\lim_{a\to 0} a^{-\frac{n}{8}-\frac{\ell}{2}} \times \mathbb{P}^a_{\mathbb{H}}\left[G(\parti;z_1^a,\ldots,z_n^a,x_1^a,\ldots,x_{\ell}^a)\right]
		\end{equation}
	exists and belongs to $(0,\infty)$. Moreover, if $\varphi$ is a conformal map from $\mathbb{H}$ onto itself such that $\varphi(x_k)\neq \infty$ for  $1\leq k\leq \ell$, {$\varphi(z_j)\neq \infty$ for  $1\leq j\leq n$}, then we have 
	\begin{equation*}
		R(\parti;\varphi(z_1),\ldots,\varphi(z_n);\varphi(x_1),\ldots,\varphi(x_{\ell}))=R(\parti;z_1,\ldots,z_n;x_1,\ldots,x_{\ell})\times \prod_{j=1}^n |\varphi'(z_j)|^{-\frac{1}{8}}\times \prod_{k=1}^{\ell} |\varphi'(x_k)|^{-\frac{1}{2}}.
	\end{equation*}
When the number of vertices is small, we have explicit expressions for $R$ up to multiplicative constants:
\begin{enumerate}[label=\textnormal{(\arabic*)}, ref=(\arabic*)]
	\item For $n=0$ and $\ell\in \{2,3\}$, there exist constants $C_4,C_5\in (0,\infty)$ such that 
	\begin{equation}\label{eqn::FK_bd_two_three}
		R(\parti_2;x_1,x_2)=C_4|x_1-x_2|^{-1},\quad R(\parti_3;x_1,x_2,x_3)=C_5|x_1-x_2|^{-\frac{1}{2}}|x_1-x_3|^{-\frac{1}{2}}|x_2-x_3|^{-\frac{1}{2}}.
	\end{equation}
As a consequence, we have the factorization formula
\begin{equation*}
R(\parti_3;x_1,x_2,x_3)=\frac{C_5}{C_4^{\frac{3}{2}}}\sqrt{R(\parti_2;x_1,x_2)\times R(\parti_2;x_1,x_3)\times R(\parti_2;x_2,x_3)}.
\end{equation*}
\item For $n=\ell=1$, $x_1=0$ and $z_1=re^{\ii\theta}=x+\ii y\in \mathbb{H}$, there exists a constant $C_6\in(0, \infty)$ such that 
\begin{equation*}
	R(\parti;0;z)=C_6\frac{y^{\frac{3}{8}}}{|z|} =C_6\frac{\left(\sin\theta\right)^{\frac{3}{8}}}{r^{\frac{5}{8}}}.
\end{equation*}
\end{enumerate}
\end{theorem} 

The normalization factor in~\eqref{eqn::def::R} is related to the boundary one-arm exponent for the FK-Ising model (see~\cite[Theorems~1 and 2]{PolyFKIsing}).
As part of the proof of~\eqref{eqn::scaling_boundary}, we will derive
\begin{equation} \label{eqn::scaling_boundary}
	\lim_{a\to 0}a^{-\frac{1}{2}} \mathbb{P}^a_{\mathbb{H}}\left[0\longleftrightarrow \partial B_1(0)\right]=C_7
\end{equation}
 using Smirnov's FK-Ising fermionic observable (see~\cite{SmirnovHolomorphicFermion}).
 We note that, without~\eqref{eqn::scaling_boundary}, one can still obtain a result like~\eqref{eqn::def::R}, but with $a^{\frac{1}{2}}$ replaced by $\mathbb{P}_{\mathbb{H}}^a\left[0\longleftrightarrow \partial B_1(0)\right]$. This follows from the observation that
 \begin{equation}\label{eqn::scaling_bd}
 	\lim_{a\to 0}\frac{\mathbb{P}_{\mathbb{H}}^a\left[0\longleftrightarrow \partial B_{\epsilon}(0)\right]}{\mathbb{P}_{\mathbb{H}}^a\left[0\longleftrightarrow \partial B_1(0)\right]}=\epsilon^{-\frac{1}{2}},\quad \forall \epsilon>0,
 \end{equation}
which can be derived using the boundary one-arm exponent obtained in~\cite[Theorems~1 and 2]{PolyFKIsing} and the argument in~\cite[Proof of Proposition~4.9]{GarbanPeteSchrammPivotalClusterInterfacePercolation}. 


\subsection{Organization of the rest of the paper and outlook}
In Section~\ref{sec::pre}, we collect some known results that will be used in the proofs of the main results of the paper. In Section~\ref{sec::con::int}, we study the connection probabilities of points in the bulk and prove Theorem~\ref{thm::cvg_proba}. In Section~\ref{sec::con::bou}, we study connection probabilities of points that can be either in the bulk or on the boundary, and prove Theorem~\ref{thm::main_boundary}. The paper ends with an appendix dedicated to the Ising model in a domain with a boundary, in which we provide explicit formulas for some Ising boundary spin correlations.

Theorems~\ref{thm::cvg_proba} and~\ref{thm::main_boundary} consider connection probabilities between points at fixed Euclidean distance from each other, which are related to correlation functions of the Ising spin (magnetization) field.
The Ising energy correlations on the lattice can also be expressed in terms of point-to-point connection probabilities in the FK-Ising model via the Edwards-Sokal coupling.
It would be interesting if one could give a more geometric approach to establish the conformal covariance of Ising energy correlations, as explained in Section~\ref{sec::bac::mot}. A fundamental difference is that, in this case, one would need to consider vertices that are a finite number of lattice spaces apart.

In recent works~\cite{CamiaFengLogPercolationMath,CamiaFengLogPercolationPhysics} on critical Bernoulli site percolation on the triangular lattice, we studied the asymptotic behavior of certain limiting connection probabilities as two points get close to each, identifying the presence of a logarithmic correction to the leading-order power-law behavior.
It would be interesting to extend that analysis to the FK model.
However, the arguments and ideas in~\cite{CamiaFengLogPercolationMath,CamiaFengLogPercolationPhysics} are not sufficient to deal with the critical random-cluster models with $q\neq 1$ (even if we assume the convergence of interfaces towards $\mathrm{CLE}_{\kappa}$) because, when $q\neq 1$, one loses independence and, in particular, one needs to consider the influence of the boundary on the states of edges in the bulk.

In future work, we plan to study the Ising energy field and explore a new proof of conformal covariance of Ising energy correlations at criticality based on the convergence of interfaces in the critical FK-Ising model towards $\mathrm{CLE}_{16/3}$, with the hope that it can be generalized to deal with the Potts model with other values of $q$ (assuming the convergence of interfaces towards $\mathrm{CLE}_{\kappa}$ for the corresponding critical random-cluster model).

\section{Preliminaries} \label{sec::pre}
In this section, we collect some known results that will be used in various places of our proofs. The first is the convergence of FK-Ising interfaces in domains with Dobrushin boundary conditions towards $\mathrm{SLE}_{16/3}$ curves, which was proven in a celebrated group effort summarized in~\cite{CDHKS:Convergence_of_Ising_interfaces_to_SLE}. 
The second is the convergence of FK-Ising loop ensembles towards $\mathrm{CLE}_{16/3}$ given in~\cite{KemppainenSmirnovFullLimitFKIsing,KemppainenSmirnovBoundaryTouchingLoopsFKIsing}. The third is Wu's classic result~\cite{TTWu} on the scaling limit of Ising two-point correlation functions. The last one is the convergence of Smirnov's FK-Ising fermionic observable~\cite{SmirnovHolomorphicFermion}. As explained in the introduction, Wu's result on the Ising two-point function and the convergence of Smirnov's FK-Ising observable are only used to figure out the exact orders of the normalization factors in Theorems~\ref{thm::cvg_proba} and~\ref{thm::main_boundary}.

\subsection{Conformal invariance of interfaces and loop ensembles}
\subsubsection*{Dobrushin domains}

A \textit{discrete Dobrushin domain} is a simply connected subgraph of $\Z^2$, or $a \Z^2$, with two marked boundary points $x_1, x_2$ in counterclockwise order, whose precise definition is given below.

Firstly, we define the \textit{medial Dobrushin domain}.
{Edges are oriented in such a way that the four edges around a vertex of $\mathbb{Z}^2$ (respectively, $(\mathbb{Z}^2)^{\bullet}$) form a circuit that winds around the vertex clockwise (resp., counterclockwise).}
Let $x_1^\diamond, x_{2}^\diamond$ be distinct medial vertices. Let $(x_1^\diamond \, x_2^\diamond), (x_{2}^\diamond \, x_{1}^\diamond)$ be two oriented paths on $(\Z^2)^\diamond$ satisfying the following conditions {(where we use the convention $x_3^\diamond=x_1^\diamond$)}: 
\begin{itemize}[leftmargin=2em]
	
	\item  the two paths are edge-avoiding and satisfy $(x_1^\diamond x_2^\diamond)\cap (x_2^\diamond x_1^{\diamond})=\{x_1^{\diamond},x_2^{\diamond}\}$;
	
	\item  
	the infinite connected component of $(\Z^2)^\diamond\setminus \smash{\overset{2}{\underset{j=1}{\bigcup}}}
	(x_j^\diamond \, x_{j+1}^\diamond)$ lies on the right (resp., left) of the oriented path $(x_1^\diamond \, x_2^\diamond)$ (resp., $(x_2^\diamond\, x_2^\diamond)$). 
\end{itemize}
Given $\{(x_j^\diamond \, x_{j+1}^\diamond) \colon 1\leq j\leq 2\}$, the medial Dobrushin domain $(\Omega^\diamond; x_1^\diamond, x_{2}^\diamond)$ is defined 
as the subgraph of $(\Z^2)^\diamond$ induced by the vertices lying on or enclosed by the circuit obtained by concatenating $(x_1^\diamond x_2^\diamond)$ and $(x_2^\diamond x_1^\diamond)$.
For each $j \in \{1,2\}$, the \textit{outer corner}  $w_{j}^{\diamond}\in (\mathbb{Z}^2)^\diamond\setminus\Omega^\diamond$ is defined to be a medial vertex adjacent to $x_j^\diamond$, and the \textit{outer corner edge} $e_j^\diamond$ is defined to be the medial edge connecting {$x_j^{\diamond}$ and $w_{j}^{\diamond}$}.

\smallbreak

Secondly, we define the \textit{primal Dobrushin domain} 
$(\Omega;x_1, x_{2})$ induced by $(\Omega^\diamond;x_1^\diamond,x_{2}^\diamond)$ as follows: 
\begin{itemize}[leftmargin=2em]
	\item 
	the edge set $E(\Omega)$ consists of edges passing through endpoints of medial edges in 
	$E(\Omega^\diamond)\setminus (x_2^\diamond x_1^\diamond)$; 
	
	\item  
	the vertex set $V(\Omega)$ consists of endpoints of edges in $E(G)$; 
	
	\item  
	the marked boundary vertex $x_j$ is defined to be the vertex in $\Omega$ nearest to $x_j^\diamond$ for each $j=1,2$; 
	
	\item  
	the arc $(x_1 \, x_2)$ is the set of edges whose midpoints are vertices in $(x_1^\diamond \, x_2^\diamond)\cap \partial \Omega$.
\end{itemize}

Lastly, we define the \textit{dual Dobrushin {domain} }$(\Omega^{\bullet};x_1^{\bullet},x_{2}^{\bullet})$ induced by $(\Omega^\diamond; x_1^\diamond,x_{2}^\diamond)$ in a similar way. 
More precisely, $\Omega^{\bullet}$ is the subgraph of $(\Z^2)^{\bullet}$ with edge set consisting of edges passing through endpoints of medial edges in $E(\Omega^\diamond)\setminus (x_2^\diamond x_1^\diamond)$ and vertex set consisting of the endpoints of these edges. 
The marked boundary vertex $x_j^{\bullet}$ is defined to be the vertex in $V(\Omega^{\bullet})$ nearest to $x_j^\diamond$ for $j=1,2$. 
The boundary arc $(x_{2}^{\bullet} \, x_{1}^{\bullet})$ is the set of edges whose midpoints are vertices in $(x_{2}^\diamond \, x_{1}^\diamond)\cap \Omega^\diamond$.

\subsection*{Boundary conditions, loops and interfaces}
{We will consider the critical FK-Ising model on $\Omega^a$ with two types of boundary conditions:
\begin{enumerate}
    \item free boundary conditions,
    \item Dobrushin boundary conditions, that is, free on $(x_2^a x_1^a)$ and wired on $(x_1^a x_2^a)$.
\end{enumerate}
We note that the first type can be considered a degenerate case of the second, with $x_1^a=x_2^a$.}

Let $\omega\in \{0,1\}^{E(\Omega^a)}$ be a configuration of the FK-Ising model on $\Omega^a$. 
{For both types of boundary conditions mentioned above, we can draw edge-self-avoiding interfaces on $\Omega^{a,\diamond}$ using the edges of the medial lattice as follows:
\begin{itemize}
    \item each edge belongs to a unique interface,
    \item edges are connected in such a way that no interface crosses an open primal edge or open dual edge.
\end{itemize}

In the case of free boundary conditions, the edges of the medial lattice form a collection $\Gamma^a$ of loops that do not cross each other or themselves.
In the (non-degenerate) case of Dobrushin boundary conditions, in addition to loops, there is an edge-self-avoiding interface $\gamma^a$ connecting the outer corners $w_1^{a,\diamond}$ and $w_2^{a,\diamond}$ on the medial Dobrushin domain $(\Omega^{a,\diamond};x_1^{a,\diamond},x_2^{a,\diamond})$. Both $\Gamma^a$ and $\gamma^a$ have a conformally invariant scaling limit, and we will make use of this fact.}

\subsection*{Topologies and convergence of interfaces}
{In this section, we specify the topologies used to formulate the convergence of loops and interfaces and the convergence of collections of loops.}

First, as in~\cite{Cam24}, we define a distance function $\Delta$ on $\mathbb{C}\times \mathbb{C}$ given by
\begin{equation*}
	\Delta(u,v):=\inf_{\phi}\int_0^a \frac{|\phi'(t)|}{1+|\phi(t)|^2}\ud t,
\end{equation*}
where the infimum is over all differentiable curves $\phi:[0,1]\to \mathbb{C}$ with $\phi(0)=u$ and $\phi(1)=v$. Note that, if we write $\widehat{\mathbb{C}}:=\mathbb{C}\cup \{\infty\}$ and extend $\Delta$ to be a function on $\widehat{\mathbb{C}}\times \widehat{\mathbb{C}}$, then $\left(\widehat{\mathbb{C}}, \Delta\right)$ is compact. 

Second, for two planar continuous oriented curves $\gamma_1,\gamma_2:[0,1]\to \mathbb{C}$, we define
\begin{equation} \label{eqn::def_dist}
	\dist\left(\gamma_1,\gamma_2\right):=\inf_{\psi,\tilde{\psi}}\sup_{t\in[0,1]} \Delta\left(\gamma_1(\psi(t)),\gamma_2(\tilde{\psi}(t))\right),
\end{equation}
where the infimum is taken over all increasing homeomorphisms $\psi,\tilde{\psi}:[0,1]\to [0,1]$.

Third, for two sets of loops, $\Gamma_1$ and $\Gamma_2$, we define 
\begin{equation} \label{eqn::def_Dist}
	\Dist\left(\Gamma_1,\Gamma_2\right):=\inf\left\{\epsilon>0: \forall \gamma_1\in \Gamma_1,\enspace \exists \gamma_2\in \Gamma_2 \text{ s.t. }\dist(\gamma_1,\gamma_2)\leq \epsilon \text{ and vice versa}\right\}. 
\end{equation}

\begin{theorem}\label{thm::SLE}(\cite{CDHKS:Convergence_of_Ising_interfaces_to_SLE})
	Let $\Omega\subseteq \mathbb{C}$ be a simply connected domain with locally connected boundary and let $x_1,x_2\in \partial\Omega$ be $2$ distinct points. Let $(\Omega^a;x_1^a,x_2^a)$ be a sequence of primal Dobrushin {domains} satisfying: $\Omega^a$ converges to $\Omega$ under the Hausdorff metric and $x_1^a\to x_1, x_2^a\to x_2$, as $a\to 0$. Consider the critical FK-Ising model on $(\Omega^a;x_1^a,x_2^a)$ with Dobrushin boundary conditions described above. Then the interface $\gamma^a$ converges weakly, as $a\to 0$, under the topology induced by $\dist$ (see~\eqref{eqn::def_dist}) towards $\mathrm{SLE}_{16/3}$ on $\Omega$ from $x_1$ to $x_2$ (for more details on $\mathrm{SLE}$, see~\cite{LawlerConformallyInvariantProcesses} or~\cite{RohdeSchrammSLEBasicProperty}). 
\end{theorem}

\begin{theorem} \label{thm::CLE}(\cite[Theorem~1.1]{KemppainenSmirnovFullLimitFKIsing}, ~\cite[Theorem~1.1]{KemppainenSmirnovBoundaryTouchingLoopsFKIsing})
	Assume the same setup as in Theorem~\ref{thm::cvg_proba}.   Consider the critical FK-Ising model on $\Omega^a$ with free boundary conditions. Then the collection of loops $\Gamma^a$  converges weakly as $a\to0$ under the topology induced by $\Dist$ (see~\eqref{eqn::def_Dist}). We denote the limiting measure by $\mathbb{P}=\mathbb{P}_{\Omega}$. Moreover, $\mathbb{P}_{\Omega}$ is conformally invariant. For wired boundary conditions, the corresponding conclusions also hold and we denote by $\overline{\mathbb{P}}=\overline{\mathbb{P}}_{\Omega}$ the limiting measure. 
\end{theorem}

We emphasize that the hypothesis on the convergence of discrete domains is not optimal here, but the present version will be sufficient for our purposes.

\subsection{Scaling limit of two-point Ising correlation functions and FK-Ising connection probabilities} \label{sec::two_point_Ising}
{Consider the critical Ising measure on $a\mathbb{Z}^2$ and let $\mathbb{E}_{\mathbb{Z}^2}^a$ denote the corresponding expectation.}

\begin{theorem} \label{thm::two_point_Ising}(\cite{TTWu,WUTwopoint})
	Let $y_1^a, y_2^a\in a\mathbb{Z}^2$ satisfy $\lim_{a\to 0}y_1^a=0$ and $\lim_{a\to 0}y_2^a=1$. Then we have 
	\begin{equation} \label{eqn::two_point_Ising}
		\lim_{a\to 0}a^{-\frac{1}{4}}\times \mathbb{E}_{\mathbb{Z}^2}^a\left[\sigma_{y_1^a}\sigma_{y_2^a}\right]=C_8,
	\end{equation}
where $C_8>0$ is a universal constant.   
\end{theorem}

\begin{corollary} \label{coro::two_point_FK}
Let $y_1^a, y_2^a\in a\mathbb{Z}^2$ satisfy $\lim_{a\to 0}y_1^a=0$ and $\lim_{a\to 0}y_2^a=1$. Then we have
	\begin{equation*}
		\lim_{a\to 0} a^{-\frac{1}{4}}\times \mathbb{P}_{\mathbb{Z}^2}^a\left[y_1^a\longleftrightarrow y_2^a\right]=C_8,
	\end{equation*}
where $C_8$ is the same universal constant in~\eqref{eqn::two_point_Ising}. 
\end{corollary}
\begin{proof}
The Edwards-Sokal coupling (see~\cite{EdwardsSokal}) implies that
\begin{equation} \label{eqn::two_point_aux1}
	\mathbb{P}^a_{\mathbb{Z}^2}\left[y_1^a\longleftrightarrow y_2^a\right]=\mathbb{E}^a_{\mathbb{Z}^2}\left[\sigma_{y_1^a}\sigma_{y_2^a}\right].
\end{equation}
The desired conclusion follows readily from~\eqref{eqn::two_point_aux1} and Theorem~\ref{thm::two_point_Ising}.
\end{proof}

\subsection{Scaling limit of Smirnov's FK-Ising fermionic observable} \label{sec::FKObservable}
To deal with connection probabilities involving boundary vertices, Theorem~\ref{thm::two_point_Ising}, which is a main ingredient in the proof of Theorem~\ref{thm::cvg_proba}, is not sufficient.
A manifestation of this fact is that the boundary arm exponents are typically different than the interior ones.
This affects the normalization of crossing probabilities involving boundary vertices.
For this case, unable to use Theorem~\ref{thm::two_point_Ising}, we will find the exact order of the proper normalization using Smirnov's FK-Ising fermionic observable (see~\cite{SmirnovHolomorphicFermion}), as explained below.

\begin{figure}
    \centering
    \includegraphics[width=0.45\linewidth]{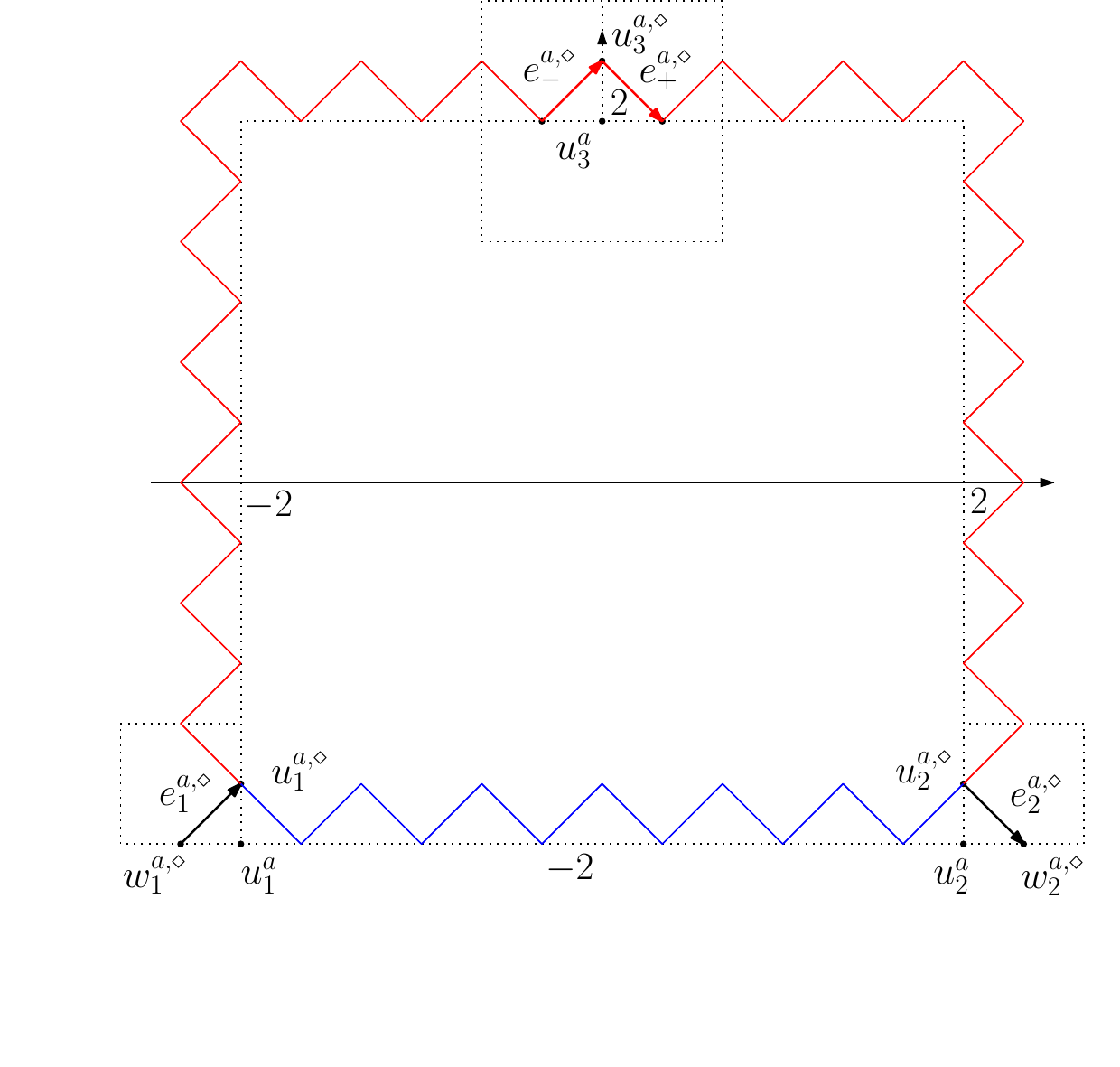}
    \caption{An illustration of the Dobrushin domain $\Box_2^a$ and the vertices, medial vertices, and (oriented) medial edges described in Section~\ref{sec::FKObservable}. Recall that the four medial edges
around a vertex of $a\mathbb{Z}^2$ (respectively, $a(\mathbb{Z}^2)^{\bullet}$) form a circuit that winds around the vertex clockwise (resp.,
counterclockwise).}
    \label{fig:FKObservable}
\end{figure}

Let $\Box_2^a=[-2,2]^2 \cap a\mathbb{Z}^2$, consider the the Dobrushin domain $(\Box_2^a;u_1^a,u_2^a)$ and recall the definitions of the medial vertices $u_{1}^{a,\diamond}$ and $u_{2}^{a,\diamond}$ adjacent to $u_{1}^{a}$ and $u_{2}^{a}$ and of the outer corners $w_{1}^{a,\diamond}$ and $w_{2}^{a,\diamond}$ adjacent to $u_{1}^{a,\diamond}$ and $u_{2}^{a,\diamond}$, respectively (see Figure~\ref{fig:FKObservable}).
\begin{proposition} \label{prop::one_arm_bd}
	Let $u_1^a$ and $u_2^a$ be the southwest and southeast corners of the box $\Box_2^a$, respectively. Let $u_3^a=\partial \Box_2^a\cap \{\ii y: y>0\}$. Consider the critical FK-Ising model on $(\Box_2^a;u_1^a, u_2^a)$ with Dobrushin boundary conditions and denote by $\mathbb{P}_*^a$ the corresponding measure. Then there exists a universal constant $C_9$ such that 
	\begin{equation} \label{eqn::sharp_boundary}
		\lim_{a\to 0} a^{-\frac{1}{2}} \times\mathbb{P}_*^a\left[u_3^a\longleftrightarrow (u_1^a u_2^a)\right]=C_9.
	\end{equation}
\end{proposition}
Note that~\eqref{eqn::sharp_boundary} gives the sharpness of the boundary one-arm exponent for the FK-Ising model. 

The proof of Proposition~\ref{prop::one_arm_bd} relies on the following  observations: (1) the choice of Dobrushin boundary conditions implies that the edges of the medial lattice form a collection $\Gamma^a$ of non-crossing loops and an edge-self-avoiding interface $\gamma^a$ parameterized from $w_{1}^{a,\diamond}$ to $w_{2}^{a,\diamond}$;
(2) denoting by $u_3^{a,\diamond}\in \partial \Box_{2}^{a,\diamond}$ the medial vertex to the north of $u_3^a$ closest to $u_3^a$
and letting {$e_{3,-}^{a,\diamond}, e_{3,+}^{a,\diamond}$} be the oriented\footnote{Recall that medial edges 
are oriented in such a way that the four edges around a vertex of $\mathbb{Z}^2$ (respectively, $(\mathbb{Z}^2)^{\bullet}$) form a circuit that winds around the vertex clockwise (resp., counterclockwise).}
medial edges around $u_3^a$ with $u_3^{a,\diamond}$ as their end vertex and beginning vertex, respectively (see Figure~\ref{fig:FKObservable}), then 
\begin{equation} \label{eqn::equi_events}
	\{u_3^a \longleftrightarrow (u_1^a u_2^a)\}=\{\gamma^a \text{ passes through }e_{3,-}^{a,\diamond} \}=\{\gamma^a \text{ passes through }e_{3,+}^{a,\diamond}\};
\end{equation}
(3) the probabilities of the latter two events in~\eqref{eqn::equi_events} can be related to the value of Smirnov's observable on the medial vertex $u_3^{a,\diamond}$. 

We interpret each oriented medial edge $e^{\diamond}$ as a complex number and define
\begin{equation*}
	\nu\left(e^{\diamond}\right):=\left(\frac{e^{\diamond}}{|e^{\diamond}|}\right)^{-\frac{1}{2}}.
\end{equation*}
Note that $\nu(e^{\diamond})$ is defined up to a sign, which we will specify when necessary. We denote by $\mathbb{E}_{*}^a$ the expectation corresponding to $\mathbb{P}_*^a$. Now let us recall the definition of FK-Ising fermionic observable given in~\cite{SmirnovHolomorphicFermion}. Recall that in the Dobrushin domain $(\Box_2^a;u_1^a,u_2^a)$, the outer corner $w_2^{a,\diamond}\in (a\mathbb{Z}^2)^{\diamond}\setminus \Box_{2}^{a,\diamond}$ is a medial vertex adjacent to $u_2^{a,\diamond}$, and the outer corner edge $e_2^{a,\diamond}$ is the medial edge connecting $u_2^{a,\diamond}$ and $w_2^{a,\diamond}$.
\begin{itemize}
	\item First, define the \textit{edge observable} on edges and outer corner edges $e$ of $\Box_{2}^{a,\diamond}$ as
\begin{align*}
	F^{a}(e) :=  
	\nu(e_{2}^{a,\diamond}) 
	\; \E^{a}_{*} \Big[ \one \{e\in\gamma^{a}\} 
	\exp\Big(-\frac{\ii}{2} W_{\gamma^{a}} \big( e_{2}^{\delta,a} ,e \big) \Big) \Big] ,
\end{align*}
where $e_{2}^{a,\diamond}$ is the oriented outer corner edge connecting to $w_{2}^{a,\diamond}$ and oriented to have $w_{2}^{a,\diamond}$ as its end vertex, 
$W_{\gamma^a}\big(e_{2}^{a,\diamond}, e\big) {\in \R}$ is the winding number from $w_{2}^{a,\diamond}$ to $e$ 
along the reversal of $\gamma^{a}$. Note that $F^a$ is only defined up to a sign. 

\item Second, we define the \textit{vertex observable} on interior vertices $z^{\diamond}$ of $\Box_{2}^{a,\diamond}$ as
\begin{align*}
	F^{a}(z^{\diamond}) := \frac{1}{2}\sum_{e^{\diamond} \sim z^{\diamond}} F^{a}(e^{\diamond}) ,	
\end{align*}
where the sum is over the four medial edges $e^{\diamond} \sim z^{\diamond}$ having $z^{\diamond}$ as an endpoint. 

\item Third, we define the \textit{vertex observable} on vertices in $\partial\Box_2^{a,\diamond} \setminus \{u_1^{a,\diamond},u_2^{a,\diamond}\}$ as follows. 
For any $z^{a,\diamond} \in \partial\Box_2^{a,\diamond} \setminus \{u_1^{a,\diamond},u_2^{a,\diamond}\}$, let $e_-^{a,\diamond}, e_+^{a,\diamond} \in \partial\Box_2^{a,\diamond} \setminus \{u_1^{a,\diamond},u_2^{a,\diamond}\}$ be the oriented medial edges having $z^{a,\diamond}$ as their end vertex and beginning vertex, respectively. Set 
\begin{align*} 
	F^{a}(z^{\diamond}) =
	\begin{cases}
		\sqrt{2} \exp(-\ii\frac{\pi}{4})F^{a}(e_+^{a,\diamond}) + \sqrt{2} \exp(\ii\frac{\pi}{4})F^a(e_-^{a,\diamond}), & \text{if $z^{a,\diamond}$ $\in (u_1^{a,\diamond}u_2^{a,\diamond})$} , \\[.5em]
		\sqrt{2} \exp(-\ii\frac{\pi}{4})F^a(e_-^{a,\diamond}) + \sqrt{2} \exp(\ii\frac{\pi}{4})F^a(e_+^{a,\diamond}), & \text{if $z^{a,\diamond}$ $\in (u_2^{a,\diamond}u_1^{a,\diamond})$}.
	\end{cases}
\end{align*} 
\end{itemize}

\begin{lemma}\label{lem::one_arm_bd_aux1}
	With an appropriate choice of the sign of $\nu(e_2^{a,\diamond})$, we have 
	\begin{equation} \label{eqn::value_proba}
		F^a(u_3^{a,\diamond})= 2\sqrt{2} \cos\left(\frac{\pi}{8}\right)\times \mathbb{P}_*^a \left[\gamma^a \text{passes through } {e_{3,-}^{a,\diamond}} \right]=2\sqrt{2} \cos\left(\frac{\pi}{8}\right)\times\mathbb{P}_*^a\left[u_3^a\longleftrightarrow (u_1^a u_2^a)\right]. 
	\end{equation}
\end{lemma}
\begin{proof}
	The first equal sign in~\eqref{eqn::value_proba} follows from~\cite[Eq.~(3.25)]{FPWFKIsing} and the observation that the winding number $W_{\gamma^a}(e_2^{a,\diamond}, e_{3,-}^{a,\diamond})$ is the same for all FK-Ising configurations. The second equal sign follows from~\eqref{eqn::equi_events}.
\end{proof}

It is a celebrated result in~\cite{SmirnovHolomorphicFermion} that, {as $a \to 0$, the function} $2^{-1/4}a^{-\frac{1}{2}}F^a(\cdot)$ converges locally uniformly towards an explicit holomorphic function on $[-2,2]^2$.
Since the boundary of our discrete domain $\Box_2^{a}$ is flat near $u_3^a$, we also have the convergence of $2^{-1/4}a^{-\frac{1}{2}}F^a(u_3^{a,\diamond})$.

\begin{lemma} \label{lem::one_arm_bd_aux2}
	We have the convergence
	\begin{equation*}
		\lim_{a\to 0} 2^{-1/4}a^{-\frac{1}{2}}|F^a(u_3^{a,\diamond})|=|\phi(\ii;[-2,2]^2;-1-\ii,1-\ii)|,
	\end{equation*}
where $\phi(\cdot;[-2,2]^2;-1-\ii,-1+\ii)$ is the unique (up to a sign) holomorphic function defined in~\cite[Proposition~3.6 and Remark~3.9]{FPWFKIsing}.
\end{lemma}
\begin{proof}
The boundary of $\Box_2^{a}$ near $u_3^a$ satisfies the regularity assumption in~\cite[Definition~3.14]{ChelkakIzyHolomorphic}. Thus, we can repeat the argument in~\cite[Proof of Lemma~4.8]{ChelkakIzyHolomorphic} to obtain the desired convergence. 
\end{proof}

\begin{proof}[Proof of Proposition~\ref{prop::one_arm_bd}]
	The desired conclusion follows immediately from Lemmas~\ref{lem::one_arm_bd_aux1} and \ref{lem::one_arm_bd_aux2}.
\end{proof}

\section{Connection probabilities of interior vertices} \label{sec::con::int}
\subsection{One-arm event coupling and the spatial mixing property}
For $0< r<R$, we denote by $\mathcal{A}_{r,R}(z)$ the event $\{\partial B_r(z)\longleftrightarrow \partial B_{R}(z)\}$ and by $\mathcal{O}_{r,R}(z)$ the event that there exists an open circuit surrounding $z$ inside $B_{R}(z)\setminus B_r(z)$.  If $A\subseteq \mathbb{C}$, we define 
\begin{equation} \label{eqn::def_ecu}
	d(z,A):=\inf_{w\in A} |z-w|. 
\end{equation}

The following lemma is an analog of~\cite[Lemma~2.1]{Cam24} for the FK-Ising model.

\begin{lemma} \label{lem::separation}
	Let $\Omega^a\subseteq a\mathbb{Z}^2$ and $z\in \Omega^a$. Let $\epsilon>0$ satisfy $\epsilon<\frac{d(\partial\Omega^a,z)}{10}$. Consider the critical FK-Ising measure $\mathbb{P}^a_{\pi}$ on $\Omega^a$ with arbitrary boundary condition $\pi$. Then for any $\epsilon>\delta>\eta>a$, there exists a coupling, $\mathbb{P}^{a}_{\eta,\delta}$, between $\tilde{\Lambda}^a\sim\mathbb{P}^a_{\pi}\left[\cdot | z\longleftrightarrow \partial B_{\epsilon}(z)\right]$ and $\widehat{\Lambda}^a\sim\mathbb{P}^a_{\pi}\left[\cdot| \mathcal{A}_{\eta,\epsilon}(z)\right]$, and an event $\mathcal{S}$, such that 
	\begin{equation*}
	\widehat{\mathcal{O}}_{\eta,\delta}(z)\subseteq \mathcal{S},
	\end{equation*}
where $	\widehat{\mathcal{O}}_{\eta,\delta}(z)$ denotes the event that there exists an open circuit surrounding $z$ inside $B_\delta(z)\setminus B_{\eta}(z)$ in $\widehat{\Lambda}^a$, and such that if $\mathcal{S}$ happens, then status of edges outside $B_{\delta}(z)$ is the same under both configurations $\tilde{\Lambda}^a$ and $\widehat{\Lambda}^a$. 
In particular, there exist universal constants $c_1,c_2\in (0,\infty)$ such that 
\begin{equation*}
	\mathbb{P}^{a}_{\eta,\delta}[\mathcal{S}]\geq 1-c_1\left(\frac{\eta}{\delta}\right)^{c_2}.
\end{equation*}
\end{lemma}
\begin{proof}
    The proof is essentially the same as that of~\cite[Lemma~2.1]{Cam24}. 
 The same strategy works here because the proof of~\cite[Lemma~2.1]{Cam24} is based on the FKG inequality and RSW estimates.
 Like percolation, the FK-Ising model also satisfies the FKG inequality (see, e.g.,~\cite{BurtonMeaneDensityUniquenessPercolation}) and RSW estimates (as shown in~\cite{DuminilCopinHonglerNolinRSWFKIsing}).
\end{proof}

We denote by $\mu_{G}^{\pi}$ the critical FK-Ising measure on $G\subseteq \mathbb{Z}^2$ with boundary condition $\pi$. For $N\geq 1$, write $\Box_N=[-N,N]^2\cap \mathbb{Z}^2$. 
  We will also use the ``spatial mixing property'' of the critical FK-Ising model:
\begin{lemma} \label{lem::spatial}
	There exist two universal constants $c_3, c_4\in (0,\infty)$ such that, for any $10N<M$, any boundary conditions $\tau,\pi$ on $\partial\Box_{M}$ and any event $\mathcal{A}$ that depends only on state of edges inside $\Box_{N}$, we have 
	\begin{equation*}
		|\mu_{\Box_{M}}^{\pi}(\mathcal{A})-\mu_{\Box_{M}}^{\tau}(\mathcal{A})|\leq c_3\left(\frac{N}{M}\right)^{c_4}\times \mu_{\Box_{M}}^{\pi}(\mathcal{A}).
	\end{equation*}
\end{lemma}
\begin{proof}
See~\cite[Proposition~5.11]{DuminilCopinHonglerNolinRSWFKIsing}.
\end{proof}
\subsection{Proof of Theorem~\ref{thm::cvg_proba}}
In the rest of the paper, let $\{\delta_m\}_{m=1}^{\infty}$ be a decreasing sequence such that $\lim_{m\to \infty}\delta_m=0$.
\subsubsection{Reduction to CLE conditional probabilities}
 We use the same strategy as in~\cite{Cam24} to prove the following result:
\begin{lemma} \label{lem::cvg_proba_aux1}
	Let $\parti$ be a partition of $\{1,2,\ldots,n\}$ that contains no singletons. Then, for any $\epsilon>0$ with 
	\begin{equation*}
		\epsilon<\frac{\min \{\min_{j\neq k}|z_j-z_k|,\min_{1\leq j\leq n}d(\Omega ,z_j)\}}{100},
	\end{equation*}
we have the convergence
\begin{equation} \label{eqn::cvg_proba_aux1}
	\lim_{a \to 0}\frac{\mathbb{P}^a\left[G(\parti;z_1^a,\ldots,z_{n}^a)\right]}{\mathbb{P}^a\left[z_i^a\longleftrightarrow \partial B_{\epsilon}(z_j^a), \enspace 1\leq j\leq n\right]}=\mathbb{P}\left[G(\parti;z_1,\ldots,z_n) | z_j\longleftrightarrow \partial B_{\epsilon}(z_j),\enspace 1\leq j\leq n\right],
\end{equation}
where  the right hand side of~\eqref{eqn::cvg_proba_aux1}, which belongs to $(0,\infty)$, can be defined in terms of conditional crossing probabilities as in~\eqref{eqn::def_lemaux1_aux1} and~\eqref{eqn::def_lemaux1_aux2} below for $\parti_n=(\{1,2,\ldots,n\})$, and in~\eqref{eqn::P_def_4} below for $\parti=(\{1,2\},\{3,4\})$. For general $\parti$, the quantity 
\[\mathbb{P}\left[G(\parti;z_1,\ldots,z_n) | z_j\longleftrightarrow \partial B_{\epsilon}(z_j),\enspace 1\leq j
\leq n\right]\]
 can be defined analogously.
\end{lemma}
By standard RSW arguments (see, e.g., the proofs of Lemmas 2.1 and 2.2 of~\cite{CamiaNewman2009ising}), there exists a constant $c>0$, independent of $a$, such that 
\begin{equation}
c<\frac{\mathbb{P}^a\left[G(\parti;z_1^a,\ldots,z_{n}^a)\right]}{\mathbb{P}^a\left[z_j^a\longleftrightarrow \partial B_{\epsilon}(z_j^a), \enspace 1\leq j\leq n\right]}= \mathbb{P}^a\left[G(\parti;z_1^a,\ldots,z_{n}^a)| z_j^a\longleftrightarrow \partial B_{\epsilon}(z_j^a),\enspace 1\leq j\leq  n\right]\leq 1.
\end{equation}
Thus, any subsequential limit of $\mathbb{P}^a\left[G(\parti;z_1^a,\ldots,z_{n}^a)\right]/\mathbb{P}^a\big[z_j^a\longleftrightarrow \partial B_{\epsilon}(z_j^a), \enspace 1\leq j\leq n\big]$ must belong to $(0,\infty)$. We will prove Lemma~\ref{lem::cvg_proba_aux1} in two steps: first, we will prove it for general $n$ and $\parti=\parti_n=(\{1,2,\ldots,n\})$, that is, when all vertices belong to the same open cluster; then, we will give the proof for $n=4$ and $\parti=(\{1,2\},\{3,4\})$. All other cases can be treated similarly.

\begin{proof}[Proof of Lemma~\ref{lem::cvg_proba_aux1} for $\parti=\parti_n$]
Since the strategy is essentially the same as in~\cite[Proof of Theorem~1.1]{Cam24}, we only sketch the proof here. 


 For fixed $\delta_m$, choose $\eta>0$ such that $a<\eta<\delta_m<\epsilon$. Thanks to Lemma~\ref{lem::separation}, there exists a coupling, $\mathbb{P}^{a}_{\eta,\delta_m}$, between configurations $\tilde{\Lambda}^a$ and $\widehat{\Lambda}^a$ distributed to $\mathbb{P}^a\left[\;\cdot\; |z_j^a\longleftrightarrow \partial B_{\epsilon}(z_j^a),\enspace 1\leq j\leq n\right]$ and $\mathbb{P}^a\left[\;\cdot \;| \mathcal{A}_{\eta,\epsilon}(z_j^a),\enspace 1\leq j\leq n\right]$, respectively, and an event $\mathcal{S}_a$ such that 
 \begin{equation*}
 	\cap_{i=1}^n \widehat{\mathcal{O}}_{\eta,\delta_m}(z_j^a)\subseteq\mathcal{S}_a,
 \end{equation*}
where $\widehat{\mathcal{O}}_{\eta,\delta_m}(z_j^a)$ denotes the event that there exists an open circuit surrounding $z_j^a$ inside $B_{\delta_m}(z_j^a)\setminus B_{\eta}(z_j^a)$ in $\widehat{\Lambda}^a$, and such that if $\mathcal{S}_a$ happens, then the states of the edges outside $\cup_{j=1}^n B_{\delta_m}(z_j^a)$ are the same in $\tilde{\Lambda}^a$ and $\widehat{\Lambda}^a$. 
	Thanks to RSW estimates, we have 
	\begin{equation} \label{eqn::S_a_near1}
		\mathbb{P}^a_{\eta,\delta_m}\left[\mathcal{S}_a\right]\geq 1-nc_1\left(\frac{\eta}{\delta_m}\right)^{c_2},
	\end{equation}
	where $c_1$ and $c_2$ are constants in Lemma~\ref{lem::spatial}.
	
	Note that 
	\begin{align}
	&	\mathbb{P}^a\left[\cap_{j\neq k}\left(\left\{ B_{\delta_m}(z_j^a)\longleftrightarrow B_{\delta_m}(z_k^a)\right\}\cap \mathcal{O}_{\delta_m,\epsilon}(z_j^a)\cap \mathcal{O}_{\delta_m,\epsilon}(z_k^a)\right)| z_j^a\longleftrightarrow\partial B_{\epsilon}(z_j^a),\enspace 1\leq j\leq n\right] \label{eqn::reduc_1_aux1}\\
		\leq &\mathbb{P}^a\left[z_j^a\longleftrightarrow z_k^a,\enspace 1\leq j< k\leq n| z_j^a\longleftrightarrow\partial B_{\epsilon}(z_j^a),\enspace 1\leq j\leq n\right]\notag\\
		\leq& \mathbb{P}^a\left[ B_{\delta_m}(z_j^a)\longleftrightarrow \partial B_{\delta_m}(z_k^a),\enspace 1\leq j< k\leq n| z_j^a\longleftrightarrow\partial B_{\epsilon}(z_j^a),\enspace 1\leq j\leq n \right]. \notag
	\end{align}
On the one hand,  one can show that 
\begin{align*}
	&\limsup_{a\to 0} \mathbb{P}^a\left[ B_{\delta_m}(z_j^a)\longleftrightarrow  B_{\delta_m}(z_k^a),\enspace 1\leq j< k\leq n| z_j^a\longleftrightarrow\partial B_{\epsilon}(z_j^a),\enspace 1\leq j\leq n\right]\\
	\leq & \mathbb{P}\left[B_{\delta_m}(z_j)\longleftrightarrow B_{\delta_m}(z_k),\enspace 1\leq j<k\leq n| \mathcal{A}_{\eta,\epsilon}(z_k),\enspace 1\leq j\leq n\right]+\limsup_{a\to 0}\left(1-\mathbb{P}^a_{\eta,\delta_m}\left[\mathcal{S}_a\right]\right).
\end{align*}
Thanks to~\eqref{eqn::S_a_near1}, letting $\eta\to 0$ (along some subsequence $\{\eta_r\}_{r=1}^{\infty}$) yields
\begin{align*}
	&\limsup_{a\to 0} \mathbb{P}^a\left[ B_{\delta_m}(z_j^a)\longleftrightarrow  B_{\delta_m}(z_k^a),\enspace 1\leq j< k\leq n| z_j^a\longleftrightarrow\partial B_{\epsilon}(z_j^a),\enspace 1\leq j\leq n\right]\\
	\leq &\lim_{r\to \infty} \mathbb{P}\left[B_{\delta_m}(z_j)\longleftrightarrow B_{\delta_m}(z_k),\enspace 1\leq j<k\leq n| \mathcal{A}_{\eta_r,\epsilon}(z_j),\enspace 1\leq j\leq n\right].
\end{align*}
Similarly, one can also show that 
\begin{align*}
	&\liminf_{a\to 0} \mathbb{P}^a\left[ B_{\delta_m}(z_j^a)\longleftrightarrow  B_{\delta_m}(z_k^a),\enspace 1\leq j< k\leq n| z_j^a\longleftrightarrow\partial B_{\epsilon}(z_j^a),\enspace 1\leq j\leq n\right]\\
	\geq &\lim_{r\to \infty} \mathbb{P}\left[B_{\delta_m}(z_j)\longleftrightarrow B_{\delta_m}(z_k),\enspace 1\leq j<k\leq n| \mathcal{A}_{\eta_r,\epsilon}(z_j),\enspace 1\leq j\leq n\right].
\end{align*}
Thus, we have 
\begin{align}
	&	\lim_{a\to 0} \mathbb{P}^a\left[ B_{\delta_m}(z_j^a)\longleftrightarrow  B_{\delta_m}(z_k^a),\enspace 1\leq j< k\leq n| z_j^a\longleftrightarrow\partial B_{\epsilon}(z_j^a),\enspace 1\leq j\leq n\right]\notag\\
	=&\lim_{\eta\to 0} \mathbb{P}\left[B_{\delta_m}(z_j)\longleftrightarrow B_{\delta_m}(z_k),\enspace 1\leq j<k\leq n| \mathcal{A}_{\eta,\epsilon}(z_j),\enspace 1\leq j\leq n\right]\notag\\
	=:& \mathbb{P}\left[B_{\delta_m}(z_j)\longleftrightarrow B_{\delta_m}(z_k),\enspace 1\leq j<k\leq n| z_j\longleftrightarrow \partial B_{\epsilon}(z_j),\enspace 1\leq j\leq n\right]. \label{eqn::def_lemaux1_aux1}
\end{align}
Since the quantities in the above equation are decreasing in $m$, we have 
\begin{align}
	&\lim_{m\to \infty}	\lim_{a\to 0} \mathbb{P}^a\left[ B_{\delta_m}(z_j^a)\longleftrightarrow  B_{\delta_m}(z_k^a),\enspace 1\leq j< k\leq n| z_j^a\longleftrightarrow\partial B_{\epsilon}(z_j^a),\enspace 1\leq j\leq n\right]\notag\\
	=&\lim_{m\to \infty} \mathbb{P}\left[B_{\delta_m}(z_j)\longleftrightarrow B_{\delta_m}(z_k),\enspace 1\leq j<k\leq n| z_j\longleftrightarrow \partial B_{\epsilon}(z_j),\enspace 1\leq j\leq n\right]\notag\\
	=:&\mathbb{P}\left[z_j\longleftrightarrow z_k, \enspace 1\leq j<k\leq n| z_j\longleftrightarrow B_{\epsilon}(z_j),\enspace 1\leq j\leq n\right]. \label{eqn::def_lemaux1_aux2}
\end{align}

On the other hand, for the term in~\eqref{eqn::reduc_1_aux1}, one can use (thanks to the FKG inequality and RSW estimates) 
\begin{equation*}
	\lim_{m\to \infty}\liminf_{a\to0} \mathbb{P}^a\left[\cap_{j=1}^n \mathcal{O}_{\delta_m,\epsilon}|\cap_{j\neq k}\left(\{B_{\delta_m}(z_j)\longleftrightarrow B_{\delta_m}(z_k)\}\cap \{z_j\longrightarrow B_{\epsilon}(z_j)\}\cap\{z_k\longleftrightarrow B_{\epsilon}(z_k)\}\right)\right]=1,
\end{equation*}
to show that 
\begin{align*}
	&\lim_{m\to \infty}\liminf_{a\to 0}\mathbb{P}^a\left[\cap_{j\neq k}\left(\left\{ B_{\delta_m}(z_j^a)\longleftrightarrow B_{\delta_m}(z_k^a)\right\}\cap \mathcal{O}_{\delta_m,\epsilon}(z_j^a)\cap \mathcal{O}_{\delta_m,\epsilon}(z_k^a)\right)| z_j^a\longleftrightarrow\partial B_{\epsilon}(z_j^a),\enspace 1\leq j\leq n\right] \\
	=& \mathbb{P}\left[z_j\longleftrightarrow z_k, \enspace 1\leq j<k\leq n| z_j\longleftrightarrow B_{\epsilon}(z_j),\enspace 1\leq j\leq n\right].
\end{align*}

Combining the observations above, we obtain the desired result. 
\end{proof}
%
%

\begin{proof}[Proof of Lemma~\ref{lem::cvg_proba_aux1} for $\parti=(\{1,2\},\{3,4\})$]
One can proceed as above to show that for $k,m\geq 1$,  
\begin{align*}
	&\mathbb{P}\left[B_{\delta_m}(z_1)\longleftrightarrow B_{\delta_{m}}(z_2), B_{\delta_m}(z_3)\longleftrightarrow B_{\delta_m}(z_4), B_{\delta_k}(z_1)\centernot{\longleftrightarrow}B_{\delta_k}(z_3)|z_j\longleftrightarrow\partial B_{\epsilon}(z_j),\enspace 1\leq j\leq 4\right]\\
	:=&\lim_{\eta\to 0}\mathbb{P}\left[B_{\delta_m}(z_1)\longleftrightarrow B_{\delta_{m}}(z_2), B_{\delta_m}(z_3)\longleftrightarrow B_{\delta_m}(z_4), B_{\delta_k}(z_1)\centernot{\longleftrightarrow}B_{\delta_k}(z_3)|\mathcal{A}_{\eta,\epsilon}(z_j),\enspace 1\leq j\leq 4\right]
\end{align*}
exists. We define
\begin{align} 
	& \mathbb{P}\left[z_1\longleftrightarrow z_2\centernot{\longleftrightarrow} z_3\longleftrightarrow z_4 |z_j\longleftrightarrow \partial B_{\epsilon}(z_j),\enspace 1\leq j\leq 4\right]\notag \\
	& \quad :=\mathbb{P}\big[\cap_{m\geq 1}\cap_{k\leq m}\{B_{\delta_m}(z_1)\longleftrightarrow B_{\delta_{m}}(z_2), B_{\delta_m}(z_3)\longleftrightarrow B_{\delta_m}(z_4), B_{\delta_k}(z_1)\centernot{\longleftrightarrow}B_{\delta_k}(z_3)\} \\
    & \qquad \qquad \qquad \qquad \qquad \vert \, z_j\longleftrightarrow \partial B_{\epsilon}(z_j),\enspace 1\leq j\leq 4\big]\notag \\
	& \quad = \lim_{k\to \infty}\lim_{m\to\infty} \mathbb{P}\big[B_{\delta_m}(z_1)\longleftrightarrow B_{\delta_{m}}(z_2), B_{\delta_m}(z_3)\longleftrightarrow B_{\delta_m}(z_4), B_{\delta_k}(z_1)\centernot{\longleftrightarrow}B_{\delta_k}(z_3) \\
    & \qquad \qquad \qquad \qquad \qquad \vert \, z_j\longleftrightarrow\partial B_{\epsilon}(z_j),\enspace 1\leq j\leq 4 \big].\label{eqn::P_def_4}
\end{align} 
We denote by $\{ B_{\delta_m}(z_1)\longleftrightarrow B_{\delta_m}(z_2)\circ B_{\delta_m}(z_3)\longleftrightarrow B_{\delta_m}(z_4)\}$ the event that, outside $\cup_{j=1}^4 B_{\delta_m}(z_j)$, there are two disjoint open clusters connecting $B_{\delta_m}(z_1)$ to $B_{\delta_m}(z_2)$ and $B_{\delta_m}(z_3)$ to $B_{\delta_m}(z_4)$, respectively.
Then one can proceed as in~\cite[Proof of Theorem~1.5]{Cam24} to show that 
\begin{align*}
	& \lim_{a\to 0}\mathbb{P}^a\left[z_1^a\longleftrightarrow z_2^a\centernot{\longleftrightarrow}z_3^a\longleftrightarrow z_4^a | z_j^a\longleftrightarrow \partial B_{\epsilon}(z_j^a),\enspace 1\leq j\leq 4\right] \\
	& \quad = \mathbb{P}\left[z_1\longleftrightarrow z_2\centernot{\longleftrightarrow} z_3\longleftrightarrow z_4 |z_j\longleftrightarrow \partial B_{\epsilon}(z_j),\enspace 1\leq j\leq 4\right],
\end{align*}
with~\cite[Eq.~(2.56)]{Cam24} replaced by
\begin{align*}
	&\mathbb{P}\left[\{ B_{\delta_m}(z_1)\longleftrightarrow B_{\delta_m}(z_2)\circ B_{\delta_m}(z_3)\longleftrightarrow B_{\delta_m}(z_4)\}\cap \{B_{\delta_k}(z_1)\longleftrightarrow B_{\delta_k}(z_3)\} | z_j\longleftrightarrow \partial B_{\epsilon}(z_j),\enspace 1\leq j\leq 4\right]\\
	& = \lim_{\eta\to 0}\mathbb{P}\left[\{ B_{\delta_m}(z_1)\longleftrightarrow B_{\delta_m}(z_2)\circ B_{\delta_m}(z_3)\longleftrightarrow B_{\delta_m}(z_4)\}\cap \{B_{\delta_k}(z_1)\longleftrightarrow B_{\delta_k}(z_3)\} | \mathcal{A}_{\eta,\epsilon}(z_j),\enspace 1\leq j\leq 4\right]\\
	& \leq\lim_{\eta\to 0}\frac{\mathbb{P}\left[ (\cup_{j=1}^4 \mathcal{F}_{\sqrt{\delta_m},L}(z_j)) \cap \{\mathcal{A}_{\eta,\delta_m}(z_j),\enspace 1\leq j\leq 4\}\right]}{\mathbb{P}[\mathcal{A}_{\eta,\epsilon}(z_j),\enspace 1\leq j\leq 4]}\\
	& \leq c \, \mathbb{P}\left[\mathcal{F}_{\sqrt{\delta_m},L}(z_1)\right] \left(\frac{\delta_m}{\epsilon}\right)^{-\frac{1}{2}}\leq \tilde{c} \left(\frac{\sqrt{\delta_m}}{L}\right)^{\frac{35}{24}-\frac{1}{100}}\times \left(\frac{\delta_m}{\epsilon}\right)^{-1/2},
\end{align*}
for some $L>0$ independent of $m$ (when $m$ is large enough), where $c,\tilde{c}\in (0,\infty)$ are two constants that do not depend on $m$, the first inequality in the last line is due to the spatial mixing property in Lemma~\ref{lem::spatial} and the exponent in Theorem~\ref{thm::two_point_Ising}\footnote{Indeed, one can replace Theorem~\ref{thm::two_point_Ising} with Lemma~\ref{lem::scaling} below.}, and where the last inequality follows from the fact that $\mathbb{P}\left[\mathcal{F}_{\sqrt{\delta_m},L}(z_1)\right]\sim \left(\frac{\sqrt{\delta_m}}{L}\right)^{\frac{35}{24}+o(1)}$ as $\frac{\sqrt{\delta_m}}{L}\to 0$, which follows from~\cite[Theorems~3 and~4]{PolyFKIsing}.
\end{proof}

\subsubsection{Proper normalization and proof of part~\ref{item::thm_cvg_proba} of Theorem~\ref{thm::cvg_proba}}
Lemma~\ref{lem::cvg_proba_aux1} provides an intermediate convergence result for crossing probabilities.
In order to obtain part~\ref{item::thm_cvg_proba} of Theorem~\ref{thm::cvg_proba}, we need to replace the denominator in~\eqref{eqn::cvg_proba_aux1}, which depends on $\Omega^a$ and $z_1^a,\ldots,z_{n}^a$, with a normalization that is independent of $\Omega^a$ and $z_1^a,\ldots,z_{n}^a$.
This is the goal of the present section. We note that such a step, which is crucial for the FK-Ising model, is not needed for percolation because in the latter model independence implies that the analog of the denominator in~\eqref{eqn::cvg_proba_aux1} can be immediately written as the $n^{th}$ power of a one-arm probability.

Recall that we denote by $\mathbb{P}^a_{\Omega}$ the critical FK-Ising measure on $\Omega^a$ with free boundary condition, and by $\mathbb{P}_{\Omega}$ the law of the limiting FK-Ising loop ensemble in $\Omega$ with free boundary condition. For $M\geq 1$, let $\Box_M=[-M,M]^2$ and $\Box_M^a=\Box_M\cap a \mathbb{Z}^2$.
{
\begin{lemma} \label{lem::norma_1}
	 With the notation of Theorem~\ref{thm::cvg_proba}, for small enough $\epsilon>0$, we have 
	\begin{align*}
		&\lim_{a\to 0}\frac{\mathbb{P}^a_{\Omega}\left[z_j^a\longleftrightarrow \partial B_{\epsilon}(z_j^a),\enspace 1\leq j\leq n\right]}{\Big(\mathbb{P}^a_{\mathbb{Z}^2}\left[0\longleftrightarrow \partial B_1(0)\right]\Big)^n}\\
		&\quad=\lim_{m\to\infty} \frac{\mathbb{P}_{\Omega}\left[z_j\longleftrightarrow \partial B_{\epsilon}(z_j),\enspace 1\leq j\leq n | z_j\longleftrightarrow \partial B_{\delta_m}(z_j),\enspace 1\leq j\leq n\right]}{\Big(\mathbb{P}_{\mathbb{C}}\left[0\longleftrightarrow \partial B_{1}(0)| 0\longleftrightarrow \partial B_{\delta_m}(0)\right]\Big)^n},
	\end{align*}
	where the equation means that the limits on both sides exist in $(0,\infty)$ and that they are equal, and where
	\begin{align*}
	\mathbb{P}_{\mathbb{C}}\left[0\longleftrightarrow \partial B_{1}(0)| 0\longleftrightarrow \partial B_{\delta_m}(0)\right]:=\lim_{M\to \infty}	\mathbb{P}_{[-M,M]^2}\left[0\longleftrightarrow \partial B_{1}(0)| 0\longleftrightarrow \partial B_{\delta_m}(0)\right]. 
	\end{align*}
\end{lemma}

\begin{proof}
	Let $M\geq 10$. We write
	\begin{align} \label{eqn::norma_1}
	\begin{split}
			\frac{\mathbb{P}^a_{\Omega}\left[z_j^a\longleftrightarrow \partial B_{\epsilon}(z_j^a),\enspace 1\leq j\leq n\right]}{\Big(\mathbb{P}^a_{\mathbb{Z}^2}\left[0\longleftrightarrow \partial B_1(0)\right]\Big)^n}=&\underbrace{\frac{\mathbb{P}^a_{\Omega}\left[z_j^a\longleftrightarrow \partial B_{\epsilon}(z_j^a),\enspace 1\leq j\leq n| z_j^a\longleftrightarrow \partial B_{\delta_m}(z_j^a),\enspace 1\leq j\leq n\right]}{\Big(\mathbb{P}^a_{\Box_M}\left[0\longleftrightarrow \partial B_1(0)| 0\longleftrightarrow \partial B_{\delta_m}(0)\right]\Big)^n}}_{T_1^{(a,m)}}\\
		&\quad\times  \underbrace{\left(\frac{\mathbb{P}^a_{\Box_M}\left[0\longleftrightarrow \partial B_1(0)| 0\longleftrightarrow \partial B_{\delta_m}(0)\right]}{\mathbb{P}^a_{\mathbb{Z}^2}\left[0\longleftrightarrow \partial B_1(0)| 0\longleftrightarrow \partial B_{\delta_m}(0)\right]}\right)^n}_{T_2^{(a,m)}}\\
		&\quad\times \underbrace{\frac{\mathbb{P}^a_{\Omega}\left[z_j^a\longleftrightarrow \partial B_{\delta_m}(z_j^a),\enspace 1\leq j\leq n\right]}{\Big(\mathbb{P}^a_{\mathbb{Z}^2}[0\longleftrightarrow \partial B_{\delta_m}(0)]\Big)^n}}_{T_3^{(a,m)}}. 
	\end{split}
	\end{align}

	For the term $T_1^{(a,m)}$, one can proceed as in the proof of Lemma~\ref{lem::cvg_proba_aux1} to show that 
	\begin{align*}
		V_{m,M}:=\lim_{a\to 0} T_1^{(a,m)}=& \frac{\mathbb{P}_{\Omega}\left[z_j\longleftrightarrow \partial B_{\epsilon}(z_j),\enspace 1\leq j\leq n | z_j\longleftrightarrow \partial B_{\delta_m}(z_j),\enspace 1\leq j\leq n\right]}{\Big(\mathbb{P}_{\Box_M}\left[0\longleftrightarrow \partial B_1(0)|0\longleftrightarrow \partial B_{\delta_m}(0)\right]\Big)^n}\\
		:=&\lim_{k\to \infty} \lim_{\eta\to 0}\frac{\mathbb{P}_{\Omega}\left[\partial B_{\delta_k}(z_j)\longleftrightarrow \partial B_{\epsilon}(z_j),\enspace 1\leq j\leq n | \partial B_{\eta}(z_j)\longleftrightarrow \partial B_{\delta_m}(z_j),\enspace 1\leq j\leq n\right]}{\Big(\mathbb{P}_{\Box_M}\left[\partial B_{\delta_k}(0)\longleftrightarrow \partial B_1(0)|\partial B_{\eta}(0)\longleftrightarrow \partial B_{\delta_m}(0)\right]\Big)^n}.
	\end{align*}
	From the spatial mixing property in Lemma~\ref{lem::spatial}, we conclude that $\{V_{m,M}\}_{M=10}^{\infty}$ is a Cauchy sequence. Consequently, we can define 
	\begin{align*}
		V_m:=\lim_{M\to\infty} V_{m,M}:=\lim_{M\to \infty}\frac{\mathbb{P}_{\Omega}\left[z_j\longleftrightarrow \partial B_{\epsilon}(z_j),\enspace 1\leq j\leq n | z_j\longleftrightarrow \partial B_{\delta_m}(z_j),\enspace 1\leq j\leq n\right]}{\Big(\mathbb{P}_{\Box_M}\left[0\longleftrightarrow \partial B_1(0)|0\longleftrightarrow \partial B_{\delta_m}(0)\right]\Big)^n}.
	\end{align*}
	A direct application of RSW arguments and the FKG inequality (see, e.g., the proofs of Lemmas 2.1 and 2.2 of~\cite{CamiaNewman2009ising}) implies that there exist two constants $c_3\, c_3\in (0,\infty)$ that do not depend on $m$ such that 
	\begin{equation*}
		c_3\leq V_m\leq c_4,
	\end{equation*}
	which implies that any subsequential limit of the sequence $\{V_m\}_{m=1}^{\infty}$ must belong to $(0,\infty)$. Let $V$ be any subsequential limit of $\{V_m\}_{m=1}^{\infty}$.
	
	For the terms $T_2^{(a,m)}$ and $T_3^{(a,m)}$, it follows from the spatial mixing property in Lemma~\ref{lem::spatial} that 
	\begin{equation*}
		\lim_{m\to \infty}\lim_{a\to 0} T_2^{(a,m)} =1,\quad \lim_{m\to \infty}\lim_{a\to 0} T_3^{(a,m)}=1. 
	\end{equation*}
	
	Combining these observations with~\eqref{eqn::norma_1} yields 
	\begin{align*}
		V\leq \liminf_{a\to 0}     \frac{\mathbb{P}^a_{\Omega}\left[z_j^a\longleftrightarrow \partial B_{\epsilon}(z_j^a),\enspace 1\leq j\leq n\right]}{\Big(\mathbb{P}^a_{\mathbb{Z}^2}\left[0\longleftrightarrow \partial B_1(0)\right]\Big)^n}   \leq \limsup_{a\to 0}  \frac{\mathbb{P}^a_{\Omega}\left[z_j^a\longleftrightarrow \partial B_{\epsilon}(z_j^a),\enspace 1\leq j\leq n\right]}{\Big(\mathbb{P}^a_{\mathbb{Z}^2}\left[0\longleftrightarrow \partial B_1(0)\right]\Big)^n}   \leq V, 
	\end{align*}
	which implies that $V$ is independent of the choice of subsequence and that 
	\begin{align*}
		\lim_{a\to 0} \frac{\mathbb{P}^a_{\Omega}\left[z_j^a\longleftrightarrow \partial B_{\epsilon}(z_j^a),\enspace 1\leq j\leq n\right]}{\Big(\mathbb{P}^a_{\mathbb{Z}^2}\left[0\longleftrightarrow \partial B_1(0)\right]\Big)^n} =V=\lim_{m\to \infty} V_m. 
	\end{align*}
\end{proof}

\begin{lemma}\label{lem::norma_2}
Let $y_1^a, y_2^a\in a\mathbb{Z}^2$ satisfy $\lim_{a\to 0}y_1^a=0$ and $\lim_{a\to 0}y_2^a=1$. Then
\begin{align*}
	\lim_{a\to 0} \frac{\mathbb{P}_{\mathbb{Z}^2}^a\left[y_1^a\longleftrightarrow y_2^a\right]}{\Big(\mathbb{P}_{\mathbb{Z}^2}^a\left[0\longleftrightarrow \partial B_1(0)\right]\Big)^2}=C,
\end{align*}
for some constant $C\in (0,\infty)$.
\end{lemma}
\begin{proof}
	One can proceed as in the proof of Lemma~\ref{lem::norma_1} to show that
	\begin{align*}
		\lim_{a\to 0} \frac{\mathbb{P}_{\mathbb{Z}^2}^a\left[y_1^a\longleftrightarrow y_2^a\right]}{\Big(\mathbb{P}_{\mathbb{Z}^2}^a\left[0\longleftrightarrow \partial B_1(0)\right]\Big)^2}=&\lim_{m\to\infty} \frac{\mathbb{P}_{\mathbb{C}}\left[0\longleftrightarrow 1 | 0\longleftrightarrow \partial B_{\delta_m}(0) ,\,  1\longleftrightarrow \partial B_{\delta_m}(1)\right]}{\Big(\mathbb{P}_{\mathbb{C}}\left[0\longleftrightarrow \partial B_1(0)| 0\longleftrightarrow \partial B_{\delta_m}(0)\right]\Big)^2}=:C\in (0,\infty),
	\end{align*}
	where 
	\begin{align*}
	\frac{\mathbb{P}_{\mathbb{C}}\left[0\longleftrightarrow 1 | 0\longleftrightarrow \partial B_{\delta_m}(0) ,\,  1\longleftrightarrow \partial B_{\delta_m}(1)\right]}{\Big(\mathbb{P}_{\mathbb{C}}\left[0\longleftrightarrow \partial B_1(0)| 0\longleftrightarrow \partial B_{\delta_m}(0)\right]\Big)^2}:=\lim_{M\to \infty}\frac{\mathbb{P}_{\Box_M}\left[0\longleftrightarrow 1 | 0\longleftrightarrow \partial B_{\delta_m}(0) ,\,  1\longleftrightarrow \partial B_{\delta_m}(1)\right]}{\Big(\mathbb{P}_{\Box_M}\left[0\longleftrightarrow \partial B_1(0)| 0\longleftrightarrow \partial B_{\delta_m}(0)\right]\Big)^2}.
	\end{align*}
	This completes the proof. 
\end{proof}


With Lemmas~\ref{lem::cvg_proba_aux1}-\ref{lem::norma_2} and Corollary~\ref{coro::two_point_FK} at hand, the proof of part~\ref{item::thm_cvg_proba} of Theorem~\ref{thm::cvg_proba} is straightforward.
\begin{proof}[Proof of part~\ref{item::thm_cvg_proba} of Theorem~\ref{thm::cvg_proba}]
	Let $\epsilon\in (0,\frac{\min \{\min_{j\neq k}|z_j-z_k|,\min_{1\leq j\leq n}d(\Omega ,z_j)\}}{100}) $ and $y_1^a,y_2^a\in a\mathbb{Z}^2$ satisfy $\lim_{a\to 0} y_1^a=0$, $\lim_{a\to 0}y_2^a=1$. Write
	\begin{align*}
	a^{-\frac{n}{8}}\times \mathbb{P}^a_{\Omega}\left[G(\parti;z_1^a,\ldots,z_{n}^a)\right] =&\frac{\mathbb{P}^a_{\Omega}\left[G(\parti;z_1^a,\ldots,z_{n}^a)\right]}{\mathbb{P}^a_{\Omega}\left[z_j^a\longleftrightarrow \partial B_{\epsilon}(z_j^a), \enspace 1\leq j\leq n\right]}\times \frac{\mathbb{P}^a_{\Omega}\left[z_j^a\longleftrightarrow \partial B_{\epsilon}(z_j^a),\enspace 1\leq j\leq n\right]}{\Big(\mathbb{P}^a_{\mathbb{Z}^2}\left[0\longleftrightarrow \partial B_1(0)\right]\Big)^n}\\
	&\times \frac{\Big(\mathbb{P}^a_{\mathbb{Z}^2}\left[0\longleftrightarrow \partial B_1(0)\right]\Big)^n}{\Big(\mathbb{P}^a_{\mathbb{Z}^2}\left[z_1^a\longleftrightarrow z_2^a\right]\Big)^{\frac{n}{2}}} \times \Big(a^{-\frac{1}{4}}\times \mathbb{P}^a_{\mathbb{Z}^2}\left[y_1^a\longleftrightarrow y_2^a\right]\Big) ^{\frac{n}{2}}. 
	\end{align*}
	Then, as a consequence of Lemmas~\ref{lem::cvg_proba_aux1}-\ref{lem::norma_2} and Corollary~\ref{coro::two_point_FK}, we have 
	\begin{align*}
		\lim_{a\to 0} 	a^{-\frac{n}{8}}\times \mathbb{P}^a_{\Omega}\left[G(\parti;z_1^a,\ldots,z_{n}^a)\right]=&\lim_{m\to\infty} \frac{\mathbb{P}_{\Omega}\left[z_j\longleftrightarrow \partial B_{\epsilon}(z_j),\enspace 1\leq j\leq n | z_j\longleftrightarrow \partial B_{\delta_m}(z_j),\enspace 1\leq j\leq n\right]}{\Big(\mathbb{P}_{\mathbb{C}}\left[0\longleftrightarrow \partial B_{1}(0)| 0\longleftrightarrow \partial B_{\delta_m}(0)\right]\Big)^n}\\
		&\times	\big(C^{-1}C_8\big)^{\frac{n}{2}}\mathbb{P}_{\Omega}\left[G(\parti;z_1,\ldots,z_n) | z_j\longleftrightarrow \partial B_{\epsilon}(z_j),\enspace 1\leq j\leq n\right]
	 \in (0,\infty),
	\end{align*}
	where $C$ is the constant in Lemma~\ref{lem::norma_2} and $C_8$ is the constant in Theorem~\ref{thm::two_point_Ising} and Corollary~\ref{coro::two_point_FK}. This completes the proof. 
\end{proof}

\subsubsection{Proof of part~\ref{item::thm::cov} of Theorem~\ref{thm::cvg_proba}}
{\begin{lemma} \label{lem::scaling}
	For any $0<r<R$, we have 
	\begin{align*}
		\lim_{a\to 0 } \frac{\mathbb{P}^a_{\mathbb{Z}^2}\left[0\longleftrightarrow \partial B_{R}(0)\right]}{\mathbb{P}^a_{\mathbb{Z}^2}\left[0\longleftrightarrow \partial B_r(0)\right]} =\Big(\frac{r}{R}\Big)^{\frac{1}{8}}. 
	\end{align*}
\end{lemma}
\begin{proof}
Throughout this proof, we write $B_r$ for $B_r(0)$. 	It suffices to show that, for any $r>0$, we have 
	\begin{align} \label{eqn::scaling_aux0}
	\lim_{a\to 0 } \frac{\mathbb{P}^a_{\mathbb{Z}^2}\left[0\longleftrightarrow \partial B_{r}\right]}{\mathbb{P}^a_{\mathbb{Z}^2}\left[0\longleftrightarrow \partial B_1\right]} =r^{-\frac{1}{8}}. 
	\end{align}
	Without loss of generality, we may assume that $r\in (0,1)$. To simplify the notation, we write $f \asymp g$ if $f/g$ is bounded by a finite constant from above and so does $g/f$.
	
	On the one hand, according to~\cite[Proof of Theorem~2]{SchrammSheffieldWilsonConformalRadii} (the first displayed equation in the proof), 
	\begin{align*}
		\mathbb{P}_{B_1} \left[\partial B_{1}\longleftrightarrow \partial B_{\delta}\right] \asymp \delta^{\frac{1}{8}},\quad \text{as }\delta\to 0.
	\end{align*}
	Then, a direct application of RSW arguments and the FKG inequality (see, e.g., the proofs of Lemmas 2.1 and 2.2 of~\cite{CamiaNewman2009ising}), combined with  the spatial mixing property in Lemma~\ref{lem::spatial}, leads to
	\begin{align} \label{eqn::scaling_aux1}
		\mathbb{P}_{\mathbb{C}}[\partial B_1\longleftrightarrow \partial B_{\delta}]:=\lim_{M\to \infty} \mathbb{P}_{\Box_M}\left[\partial B_1\longleftrightarrow \partial B_{\delta}\right] \asymp \mathbb{P}_{B_1} \left[\partial B_1\longleftrightarrow \partial B_{\delta}\right]\asymp \delta^{\frac{1}{8}},
	\end{align}
	as $\delta\to 0$. 
	
	On the other hand, one can proceed as in the proof of Lemma~\ref{lem::cvg_proba_aux1} to show that 
	\begin{align} \label{eqn::scaling_aux2}
	\lim_{a\to 0 } \frac{\mathbb{P}^a_{\mathbb{Z}^2}\left[0\longleftrightarrow \partial B_{r}\right]}{\mathbb{P}^a_{\mathbb{Z}^2}\left[0\longleftrightarrow \partial B_1\right]} = \lim_{m\to \infty} \frac{\mathbb{P}_{\mathbb{C}}[\partial B_r\longleftrightarrow \partial B_{\delta_m}]}{\mathbb{P}_{\mathbb{C}}\left[\partial B_{1}\longleftrightarrow \partial B_{\delta_m}\right]}.
	\end{align}
	We denote by $\tau$ the quantity in~\eqref{eqn::scaling_aux2}. Since $r\in (0,1)$, we have $\lim_{m\to \infty} r^m=0$. Using the scale invariance of $\mathbb{P}_{\mathbb{C}}$, we can write
	\begin{align*}
		\mathbb{P}_{\mathbb{C}}\left[\partial B_1\longleftrightarrow \partial B_{r^m}\right]= &\frac{	\mathbb{P}_{\mathbb{C}}\left[\partial B_1\longleftrightarrow \partial B_{r^m}\right]}{	\mathbb{P}_{\mathbb{C}}\left[\partial B_1\longleftrightarrow \partial B_{r^{m-1}}\right]}\cdot \frac{	\mathbb{P}_{\mathbb{C}}\left[\partial B_1\longleftrightarrow \partial B_{r^{m-1}}\right]}{	\mathbb{P}_{\mathbb{C}}\left[\partial B_1\longleftrightarrow \partial B_{r^{m-2}}\right]}\cdots \frac{\mathbb{P}_{\mathbb{C}}[\partial B_1\longleftrightarrow \partial B_r]}{1}\\
		=& \frac{\mathbb{P}_{\mathbb{C}}\left[\partial B_1\longleftrightarrow \partial B_{r^m}\right]}{\mathbb{P}_{\mathbb{C}}\left[\partial B_r\longleftrightarrow \partial B_{r^m}\right]}\cdot \frac{\mathbb{P}_{\mathbb{C}}\left[\partial B_1\longleftrightarrow \partial B_{r^{m-1}}\right]}{\mathbb{P}_{\mathbb{C}}\left[\partial B_r\longleftrightarrow \partial B_{r^{m-1}}\right]}\cdots \frac{\mathbb{P}_{\mathbb{C}}\left[\partial B_1\longleftrightarrow \partial B_{r}\right]}{\mathbb{P}_{\mathbb{C}}\left[\partial B_r\longleftrightarrow \partial B_{r}\right]}.
	\end{align*}
 Using~\eqref{eqn::scaling_aux2} and the convergence of the Ces\`aro mean gives 
	\begin{align} \label{eqn::scaling_aux3}
	\lim_{m\to \infty} \frac{1}{m}\log \mathbb{P}_{\mathbb{C}} [\partial B_1\longleftrightarrow \partial B_{r^m}] =-\log \tau.
	\end{align}
	
	Combing~\eqref{eqn::scaling_aux1} with~\eqref{eqn::scaling_aux3} gives $\tau =r^{-\frac{1}{8}}$. This yields~\eqref{eqn::scaling_aux0} and completes the proof. 
\end{proof}} }

\begin{proof}[Proof of part~\ref{item::thm::cov} of Theorem~\ref{thm::cvg_proba}]
According to Lemmas~\ref{lem::cvg_proba_aux1} and~\ref{lem::norma_1}, for small enough $\epsilon>0$, 
	\begin{align*}
		\hat{P}(\Omega;\bs{Q};z_1,\ldots, z_n):=&\lim_{a\to 0} \Big(\mathbb{P}^a_{\mathbb{Z}^2}\left[0\longleftarrow \partial B_1(0)\right]\Big)^{-n}\times \mathbb{P}^a_{\Omega}\left[G(\mathbf{Q};z_1^a,\ldots,z_n^a)\right]\\
		=&\lim_{m\to \infty}\frac{\mathbb{P}_{\Omega}\left[z_j\longleftrightarrow \partial B_{\epsilon}(z_j),\enspace 1\leq j\leq n | z_j\longleftrightarrow \partial B_{\delta_m}(z_j),\enspace 1\leq j\leq n\right]}{\Big(\mathbb{P}_{\mathbb{C}}\left[0\longleftrightarrow \partial B_{1}(0)| 0\longleftrightarrow \partial B_{\delta_m}(0)\right]\Big)^n}\\
		&\times	\mathbb{P}_{\Omega}\left[G(\parti;z_1,\ldots,z_n) | z_j\longleftrightarrow \partial B_{\epsilon}(z_j),\enspace 1\leq j\leq n\right] \in (0,\infty). 
	\end{align*}
Thanks to Corollary~\ref{coro::two_point_FK} and Lemma~\ref{lem::norma_2},
\begin{align*}
    P(\Omega;\bs{Q};z_1,\ldots,z_n)=\left(C^{-1}C_8\right)^{\frac{n}{2}} \hat{P}(\Omega;\bs{Q};z_1,\ldots,z_n),
\end{align*}
where $C$ is the constant in Lemma~\ref{lem::norma_2} and $C_8$ is the constant in Theorem~\ref{thm::two_point_Ising} and Corollary~\ref{coro::two_point_FK}.
 Therefore, it suffices to show that the function $\hat{P}$ satisfies the conformal covariance property expressed by~\eqref{eqn::conformal_covariance}.	
	
	Let $\varphi$ be a conformal map from $\Omega$ onto some $\Omega'$. Let $1\leq j\leq n$. Write $s_j=|\varphi'(z_j)|$ and let $B_{R_j(\epsilon)}(z_j)\setminus B_{r_j(\epsilon)}(z_j)$ be the thinnest annulus that contains the symmetric difference\footnote{If $\varphi^{-1}\big(B_{\epsilon}(\varphi(z_j))\big)=B_{\epsilon/s_j}(z_j)$, we then let $R_j(\epsilon)=r_j(\epsilon)=\epsilon/s_j$.} of $\varphi^{-1}\left(B_{\epsilon}(\varphi(z_j))\right)$ and $B_{\epsilon/s_j}(z_j)$. Then we have 
	\begin{equation} \label{eqn::cov_aux0}
		\lim_{\epsilon\to 0}\frac{r_j(\epsilon)}{\epsilon}=\lim_{\epsilon\to 0}\frac{R_j(\epsilon)}{\epsilon}=\frac{1}{s_j}.
	\end{equation}
Note that
	\begin{align} \label{eqn::cov_aux1}
\begin{split}
		&\qquad\hat{P}(\Omega'; \bs{Q};\varphi(z_1),\ldots,\varphi(z_n))\\
  =&\lim_{m\to \infty}\frac{\mathbb{P}_{\Omega'}\left[\varphi(z_j)\longleftrightarrow \partial B_{\epsilon}(\varphi(z_j)),\enspace 1\leq j\leq n | \varphi(z_j)\longleftrightarrow \partial B_{\delta_m}(\varphi(z_j)),\enspace 1\leq j\leq n\right]}{\Big(\mathbb{P}_{\mathbb{C}}\left[0\longleftrightarrow \partial B_{1}(0)| 0\longleftrightarrow \partial B_{\delta_m}(0)\right]\Big)^n}\\
	&\times	\mathbb{P}_{\Omega'}\left[G(\parti;\varphi(z_1),\ldots,\varphi(z_n)) | \varphi(z_j)\longleftrightarrow \partial B_{\epsilon}(\varphi(z_j)),\enspace 1\leq j\leq n\right] \\
=&\lim_{m\to \infty}\frac{\mathbb{P}_{\Omega}\left[z_j\longleftrightarrow \partial \varphi^{-1}\big(B_{\epsilon}(\varphi(z_j))\big),\enspace 1\leq j\leq n | z_j\longleftrightarrow \partial \varphi^{-1}\big(B_{\delta_m}(\varphi(z_j))\big),\enspace 1\leq j\leq n\right]}{\Big(\mathbb{P}_{\mathbb{C}}\left[0\longleftrightarrow \partial B_{1}(0)| 0\longleftrightarrow \partial B_{\delta_m}(0)\right]\Big)^n}\\
&\times \mathbb{P}_{\Omega}\left[G(\parti;z_1,\ldots,z_n) | z_j\longleftrightarrow \partial \varphi^{-1}\big(B_{\epsilon}(\varphi(z_j))\big)\enspace 1\leq j\leq n\right]\\
=&\underbrace{\lim_{m\to \infty}\frac{\mathbb{P}_{\Omega}\left[z_j\longleftrightarrow \partial \varphi^{-1}\big(B_{\epsilon}(\varphi(z_j))\big),\enspace 1\leq j\leq n | z_j\longleftrightarrow \partial \varphi^{-1}\big(B_{\delta_m}(\varphi(z_j))\big),\enspace 1\leq j\leq n\right]}{\prod_{j=1}^n\mathbb{P}_{\mathbb{C}}\left[z_j\longleftrightarrow \partial B_{1}(z_j)| z_j\longleftrightarrow \partial \varphi^{-1}\big(B_{\delta_m}(\varphi(z_j))\big)\right]}}_{T_1}\\
&\times \underbrace{\mathbb{P}_{\Omega}\left[G(\parti;z_1,\ldots,z_n) | z_j\longleftrightarrow \partial \varphi^{-1}\big(B_{\epsilon}(\varphi(z_j))\big)\enspace 1\leq j\leq n\right]}_{T_2}\\
&\times\underbrace{\lim_{m\to\infty}\frac{\prod_{j=1}^n\mathbb{P}_{\mathbb{C}}\left[z_j\longleftrightarrow \partial B_{1}(z_j)| z_j\longleftrightarrow \partial \varphi^{-1}\big(B_{\delta_m}(\varphi(z_j))\big)\right]}{\Big(\mathbb{P}_{\mathbb{C}}\left[0\longleftrightarrow \partial B_{1}(0)| 0\longleftrightarrow \partial B_{\delta_m}(0)\right]\Big)^n}}_{T_3},
\end{split}
	\end{align}
	where we used the conformal invariance of $\mathbb{P}_{\Omega}$ (Theorem~\ref{thm::CLE}) to get the second equality.
	
We treat the terms $T_1$-$T_3$ one by one. For the term $T_1$, according to Lemma~\ref{lem::norma_1} and its proof, we have
\begin{align*}
	T_1=&\lim_{a\to 0} \frac{\mathbb{P}^a_{\Omega}\left[z_j^a\longleftrightarrow \partial \varphi^{-1}\big(B_{\epsilon}(\varphi(z_j^a))\big),\enspace 1\leq j\leq n\right]}{\Big(\mathbb{P}^a_{\mathbb{Z}^2}\left[0\longleftrightarrow \partial B_1(0)\right]\Big)^n}\\
	\geq & \lim_{a\to 0}\frac{\mathbb{P}^a_{\Omega}\left[z_j^a\longleftrightarrow \partial B_{R_j(\epsilon)}(z_j^a),\enspace 1\leq j\leq n\right]}{\Big(\mathbb{P}^a_{\mathbb{Z}^2}\left[0\longleftrightarrow \partial B_1(0)\right]\Big)^n}\\
	=& \lim_{a\to 0}\frac{\mathbb{P}^a_{\Omega}\left[z_j^a\longleftrightarrow \partial B_{r_j(\epsilon)}(z_j^a),\enspace 1\leq j\leq n\right]}{\Big(\mathbb{P}^a_{\mathbb{Z}^2}\left[0\longleftrightarrow \partial B_1(0)\right]\Big)^n} \times  \frac{\mathbb{P}^a_{\Omega}\left[z_j^a\longleftrightarrow \partial B_{R_j(\epsilon)}(z_j^a),\enspace 1\leq j\leq n\right]}{\mathbb{P}^a_{\Omega}\left[z_j^a\longleftrightarrow \partial B_{r_j(\epsilon)}(z_j^a),\enspace 1\leq j\leq n\right]}\\
	=& \lim_{m\to \infty}\frac{\mathbb{P}_{\Omega}\left[z_j\longleftrightarrow \partial B_{r_j(\epsilon)}(z_j),\enspace 1\leq j\leq n | z_j\longleftrightarrow \partial \varphi^{-1}\big(B_{\delta_m}(\varphi(z_j))\big),\enspace 1\leq j\leq n\right]}{\prod_{j=1}^n\mathbb{P}_{\mathbb{C}}\left[z_j\longleftrightarrow \partial B_{1}(z_j)| z_j\longleftrightarrow \partial \varphi^{-1}\big(B_{\delta_m}(\varphi(z_j))\big)\right]}\\
	&\times \underbrace{\lim_{a\to 0} \frac{\mathbb{P}^a_{\Omega}\left[z_j^a\longleftrightarrow \partial B_{R_j(\epsilon)}(z_j^a),\enspace 1\leq j\leq n\right]}{\mathbb{P}^a_{\Omega}\left[z_j^a\longleftrightarrow \partial B_{r_j(\epsilon)}(z_j^a),\enspace 1\leq j\leq n\right]}}_{T_4}.
\end{align*}
For the new term $T_4$, note that 
\begin{align*}
    \lim_{\epsilon\to 0}T_4=\lim_{\epsilon\to 0}\lim_{a\to 0}\frac{\prod_{j=1}^n\mathbb{P}^a_{\mathbb{Z}^2}\left[z_j^a\longleftrightarrow \partial B_{R_j(\epsilon)}(z_j^a)\right]}{\prod_{j=1}^n\mathbb{P}^a_{\mathbb{Z}^2}\left[z_j^a\longleftrightarrow \partial B_{r_j(\epsilon)}(z_j^a)\right]}=\lim_{\epsilon\to 0}  \prod_{j=1}^n\left(\frac{r_j(\epsilon)}{R_j(\epsilon)}\right)^{\frac{1}{8}}=1,
\end{align*}
where we used the spatial mixing property in Lemma~\ref{lem::spatial} to get the first equality, Lemma~\ref{lem::scaling} to get the second equality and~\eqref{eqn::cov_aux0} to get the last equality. 
Similarly, for the term $T_2$, the proof of Lemma~\ref{lem::norma_1} can be used to show that 
\begin{align*}
T_2=& \lim_{a\to 0} \frac{\mathbb{P}^a_{\Omega}\left[G(\bs{Q};z_1^a,\ldots,z_n^a)\right]}{\mathbb{P}^a_{\Omega}\left[z_j^a\longleftrightarrow \partial\varphi^{-1}\big(B_{\delta_m}(\varphi(z_j^a))\big)\right]}\\
\geq& \lim_{a\to0} \frac{\mathbb{P}^a_{\Omega}\left[G(\bs{Q};z_1^a,\ldots,z_n^a)\right]}{\mathbb{P}^a_{\Omega}\left[z_j^a\longleftrightarrow \partial B_{r_j(\epsilon)}(z_j^a),\enspace 1\leq j\leq n\right]}=\mathbb{P}_{\Omega}\left[G(\bs{Q};z_1,\ldots,z_n)| z_j\longleftrightarrow \partial B_{r_j(\epsilon)}(z_j),\enspace 1\leq j\leq n\right].
\end{align*}
The proof of Lemma~\ref{lem::cvg_proba_aux1} can be used to show that, when $m$ is large enough,
\begin{align*}
\mathbb{P}_{\mathbb{C}}\left[z_j\longleftrightarrow \partial B_1(z_j)|z_j\longleftrightarrow \partial \varphi^{-1} \big(B_{\delta_m}(\varphi(z_j))\big)\right]=&\lim_{a\to 0} \frac{\mathbb{P}^a_{\mathbb{Z^2}}\left[z_j^a\longleftrightarrow \partial B_1(z_j^a)\right]}{\mathbb{P}^a_{\mathbb{Z^2}}\left[z_j^a\longleftrightarrow \partial \varphi^{-1}\big(B_{\delta_m}(\varphi(z_j^a))\big)\right]},\quad \text{for }1\leq j\leq n,\\
\mathbb{P}_{\mathbb{C}}\left[0\longleftrightarrow \partial B_1(0)| 0\longleftrightarrow \partial B_{\delta_m}(0)\right]=&\lim_{a\to 0} \frac{\mathbb{P}^a_{\mathbb{Z}^2}\left[0\longleftrightarrow \partial B_1(0)\right]}{\mathbb{P}^a_{\mathbb{Z}^2}\left[0\longleftrightarrow \partial B_{\delta_m}(0)\right]}.
\end{align*}
Consequently, for the term $T_3$, according to the proof of Lemma~\ref{lem::cvg_proba_aux1}, 
\begin{align*}
\lim_{m\to \infty}\lim_{a\to 0} \prod_{j=1}^n\frac{\mathbb{P}^a_{\mathbb{Z}^2}\left[z^a_j\longleftrightarrow \partial B_{\delta_m}(z^a_j)\right]}{\mathbb{P}^a_{\mathbb{Z}^2} \left[z^a_j\longleftrightarrow \partial B_{r_j(\delta_m)}(z^a_j)\right]}
\leq & \; T_3=\lim_{m\to \infty}\lim_{a\to 0} \prod_{j=1}^n\frac{\mathbb{P}^a_{\mathbb{Z}^2}\left[z^a_j\longleftrightarrow \partial B_{\delta_m}(z^a_j)\right]}{\mathbb{P}^a_{\mathbb{Z}^2} \left[z^a_j\longleftrightarrow \partial \varphi^{-1}\big(B_{\delta_m}(\varphi(z_j))\big)\right]}\\
\leq &\lim_{m\to \infty}\lim_{a\to 0} \prod_{j=1}^n\frac{\mathbb{P}^a_{\mathbb{Z}^2}\left[z^a_j\longleftrightarrow \partial B_{\delta_m}(z^a_j)\right]}{\mathbb{P}^a_{\mathbb{Z}^2} \left[z^a_j\longleftrightarrow \partial B_{R_j(\delta_m)}(z^a_j)\right]}.
\end{align*}
Combining this with Lemma~\ref{lem::scaling} and~\eqref{eqn::cov_aux0} gives 
\begin{equation*}
	T_3= \prod_{j=1}^n \vert\varphi'(z_j)\vert^{-\frac{1}{8}}.
\end{equation*}

Plugging these observations into~\eqref{eqn::cov_aux1} gives 
\begin{align} \label{eqn::cov_aux2}
\begin{split}
		\hat{P}(\Omega';\bs{Q};\varphi(z_1),\ldots,\varphi(z_n))\geq\prod_{j=1}^n \vert\varphi'(z_j)\vert^{-\frac{1}{8}} \times \hat{P}(\Omega;\bs{Q};z_1,\ldots, z_n). 
\end{split}
\end{align}
Similarly, one can show that 
\begin{align} \label{eqn::cov_aux3}
	\hat{P}(\Omega';\bs{Q};\varphi(z_1),\ldots,\varphi(z_n))\leq \prod_{j=1}^n \vert\varphi'(z_j)\vert^{-\frac{1}{8}} \times \hat{P}(\Omega;\bs{Q};z_1,\ldots, z_n). 
\end{align}
Combining~\eqref{eqn::cov_aux2} with~\eqref{eqn::cov_aux3} gives the desired conformal covariance property of the function $\hat{P}$ and completes the proof. 
\end{proof}

\subsection{Proof of Theorem~\ref{thm::onepoint_boundary}}
Now we consider the FK-Ising model on $\Omega^a$ with wired boundary conditions, whose measure is denoted by $\overline{\mathbb{P}}^a_{\Omega}$. 
\begin{proof}[Proof of Theorem~\ref{thm::onepoint_boundary}]
First, one can proceed as in the proof of part~\ref{item::thm_cvg_proba} of Theorem~\ref{thm::cvg_proba} to show that, for small enough $\epsilon>0$, 
\begin{align*}
g(\Omega;z) =& \lim_{a\to 0}a^{-\frac{1}{8}}\times \overline{\mathbb{P}}^a_{\Omega}\left[z^a
\longleftrightarrow \partial \Omega^a\right]\\
 = &\big(C^{-1}C_8\big)^{\frac{1}{2}}\lim_{m\to\infty} \frac{\overline{\mathbb{P}}_{\Omega}\left[z\longleftrightarrow \partial B_{\epsilon}(z) | z\longleftrightarrow \partial B_{\delta_m}(z)\right]}{\mathbb{P}_{\mathbb{C}}\left[0\longleftrightarrow \partial B_{1}(0)| 0\longleftrightarrow \partial B_{\delta_m}(0)\right]}\times	\overline{\mathbb{P}}_{\Omega}\left[z\longleftrightarrow \partial \Omega| z\longleftrightarrow \partial B_{\epsilon}(z)\right]
\in (0,\infty),
	\end{align*}
where  $C$ is the constant in Lemma~\ref{lem::norma_2},  $C_8$ is the constant in Theorem~\ref{thm::two_point_Ising},
\begin{align*}
		\overline{\mathbb{P}}_{\Omega}\left[z\longleftrightarrow \partial B_{\epsilon}(z)| z\longleftrightarrow \partial B_{\delta_m}(z)\right]:=&\lim_{ k\to \infty}\lim_{\eta\to 0} \overline{\mathbb{P}}\left[B_{\delta_k}(z)\longleftrightarrow \partial B_{\epsilon}(z) | \mathcal{A}_{\eta,\delta_m}(z)\right],
\end{align*}
and
\begin{align*}
	\overline{\mathbb{P}}_{\Omega}\left[z\longleftrightarrow \partial \Omega| z\longleftrightarrow \partial B_{\epsilon}(z)\right]:=&\lim_{ k\to \infty}\lim_{\eta\to 0} \overline{\mathbb{P}}\left[B_{\delta_k}(z)\longleftrightarrow \partial \Omega | \mathcal{A}_{\eta,\epsilon}(z)\right].
\end{align*}

Second, one can proceed as in the proof of part~\ref{item::thm::cov} of Theorem~\ref{thm::cvg_proba} to show that, for any conformal map $\varphi:\Omega\to \Omega'$, one has
\begin{equation*}
	g(\Omega';\varphi(z))=g(\Omega;z)\times |\varphi'(z)|^{-\frac{1}{8}}.
\end{equation*}
This conformal covariance property ensures that there exists a constant $C_3\in (0,\infty)$ such that 
\begin{equation*}
	g(\Omega;z)=C_3\mathrm{rad}(z,\Omega)^{-\frac{1}{8}}.
\end{equation*}
\end{proof}

\section{Connection probabilities involving boundary vertices}  \label{sec::con::bou}
We will sketch the proof Theorem~\ref{thm::main_boundary} {for two particular cases}: first, we will treat Theorem~\ref{thm::main_boundary} for $n+\ell\geq 2$ and $\parti=\parti_{n+\ell}=(\{1,2,\ldots,n+\ell\})$, that is, all vertices belong to the same cluster; second, we will treat Theorem~\ref{thm::main_boundary} for $n=\ell=2$ and $\parti=(\{1,2\},\{3,4\})$. All other cases can be treated similarly. 

\begin{proof}[Proof of Theorem~\ref{thm::main_boundary} for $\parti=\parti_{n+\ell}$]
First, we have to show the existence of nontrivial scaling limits. Write
\begin{align*}
	&a^{-\frac{n}{8}-\frac{\ell}{2}}\times \mathbb{P}_{\mathbb{H}}^a \left[z_1^a\longleftrightarrow \cdots\longleftrightarrow z_n^a\longleftrightarrow x_1^a \longleftrightarrow \cdots \longleftrightarrow x_{\ell}^a\right] \\
	=& \underbrace{\left(a^{-\frac{n}{8}}\times \left(\mathbb{P}_{\mathbb{Z}^2}^a\left[y_1^a\longleftrightarrow y_2^a\right]\right)^{\frac{n}{2}}\right)}_{ T^a_1} \times \underbrace{\left(a^{-\frac{\ell}{2}}\times \left(\mathbb{P}_*^a\left[u_3^a\longleftrightarrow (u_1^a u_2^a)\right]\right)^{\ell}\right)}_{T^a_2}\\
&\times \underbrace{\frac{\Big(\mathbb{P}^a_{\mathbb{Z}^2}\left[0\longleftrightarrow \partial B_1(0)\right]\Big)^n}{\Big(\mathbb{P}^a_{\mathbb{Z}^2}\left[y_1^a\longleftrightarrow y_2^a\right]\Big)^{\frac{n}{2}}}}_{T_3^a}\times \underbrace{\frac{\Big(\mathbb{P}^a_{\mathbb{H}}\left[0\longleftrightarrow \partial B_1(0)\right]\Big)^{\ell}}{\Big(\mathbb{P}^a_*\left[u_3^a\longleftrightarrow (u_1^au_2^a)\right]\Big)^{\ell}}}_{T_4^a}\times \underbrace{\frac{\mathbb{P}_{\mathbb{H}}^a \left[z_1^a\longleftrightarrow \cdots\longleftrightarrow z_n^a\longleftrightarrow x_1^a \longleftrightarrow \cdots \longleftrightarrow x_{\ell}^a\right]}{\left(\mathbb{P}_{\mathbb{Z}^2}^a\left[o\longleftrightarrow \partial B_1(0)\right]\right)^{n}\times \left(\mathbb{P}_{\mathbb{H}}^a\left[0\longleftrightarrow \partial B_1(0)\right]\right)^{\ell}}}_{T^a_5}.
\end{align*}
According to Corollary~\ref{coro::two_point_FK} and Lemma~\ref{lem::norma_2}, if $y^a_1 \to 0$ and $y^a_2 \to 1$, we have 
	\begin{equation*}
		\lim_{a\to 0}T^a_1=C_8^{\frac{n}{2}},\quad \lim_{a\to 0} T_3^a=C^{-\frac{n}{2}}
	\end{equation*}
	where $C_8$ is the constant in Theorem~\ref{thm::two_point_Ising} and $C$ is the constant in Lemma~\ref{lem::norma_2}.
Moreover, thanks to Proposition~\ref{prop::one_arm_bd}, we have
\begin{equation*}
	\lim_{a\to 0} T^a_2=C_9^{\ell},
\end{equation*}
where $C_9$ is the constant in Proposition~\ref{prop::one_arm_bd}.
For the term $T_4^a$, one can proceed as in the proof of Lemma~\ref{lem::norma_1} to show that
	\begin{equation*}
		\lim_{a\to 0}T_4^a=\hat{C}^{\ell}
	\end{equation*}
	for some constant $\hat{C}\in (0,\infty)$.

It remains to treat the term $T_5^a$. We write $\mathbf{N}=\{1,2,\ldots,n\}$ and $\mathbf{L}=\{1,2,\ldots,\ell\}$. One then can proceed as in the proof of Lemma~\ref{lem::cvg_proba_aux1} to show that, for small enough $\epsilon>0$,
\begin{align*}
	&\lim_{a\to 0}T_5^a\\
	&=\mathbb{P}_{\mathbb{H}}\left[ z_1\longleftrightarrow \cdots \longleftrightarrow z_n\longleftrightarrow x_1\longleftrightarrow \cdots \longleftrightarrow x_{\ell}| z_j\longleftrightarrow \partial B_{\epsilon}(z_j), x_k\longleftrightarrow \partial B_{\epsilon}(x_k),\enspace (j,k)\in \mathbf{N}\times \mathbf{L}\right]\\
&\times \lim_{m\to \infty}\frac{ \mathbb{P}_{\mathbb{H}}\left[z_j\longleftrightarrow \partial B_{\epsilon}(z_j), x_k\longleftrightarrow \partial B_{\epsilon}(x_k),\, (j,k)\in \mathbf{N}\times \mathbf{L}| z_j\longleftrightarrow \partial B_{\delta_m}(z_j), x_k\longleftrightarrow \partial B_{\delta_m}(x_k),\, (j,k)\in \mathbf{N}\times \mathbf{L}\right]}{\Big(\mathbb{P}_{\mathbb{C}}\left[0\longleftrightarrow \partial B_1(0) | 0\longleftrightarrow \partial B_{\delta_m}(0) \right]\Big)^{n}\times \Big(\mathbb{P}_{\mathbb{H}}\left[0\longleftrightarrow \partial B_1(0)| 0\longleftrightarrow \partial B_{\delta_m}(0)\right]\Big)^{\ell}},
\end{align*}
where the conditional probabilities can be defined as in the proof of Lemma~\ref{lem::cvg_proba_aux1} and the proof of Lemma~\ref{lem::norma_1}. Combining all of these observations, one derives the existence of the limit. 

Second, one can proceed as in the proof of part (2) of Theorem~\ref{thm::cvg_proba} to get the desired conformal covariance property of the limiting function $R(\parti;z_1,\ldots,z_n;x_1,\ldots,x_{\ell})$, with the additional help of~\eqref{eqn::scaling_bd}, which replaces Lemma~\ref{lem::scaling} for the $\ell$ boundary points.

Third, thanks to the conformal covariance property, the explicit expressions for $R(\bs{Q}_n;x_1,\ldots,x_n)$ with $n=2,3$ are almost immediate.

Now, let us derive the explicit expression for $R(\parti_2;z;0)$. Define
\begin{equation*}
	f(z):=\frac{|z-\overline{z}|^{\frac{3}{8}}}{|z|}.
\end{equation*}
A simple calculation shows that, for any M\"obius transformation $\varphi: \mathbb{H} \to \mathbb{H}$ with $\varphi(0)=0$, one has 
\begin{equation} \label{eqn::cov_f}
f\left(\varphi(z)\right)=f(z)\times |\varphi'(z)|^{-\frac{1}{8}}\times |\varphi'(0)|^{-\frac{1}{2}}.
\end{equation}
Combining~\eqref{eqn::cov_f} with the conformal covariance property of $R(\parti_2;z;0)$, we conclude that for any M\"obius transformation $\varphi:\mathbb{H}\to \mathbb{H}$ with $\varphi(0)=0$, one has 
\begin{equation*}
	\frac{R(\parti_2;z;0)}{f(z)}=\frac{R(\parti;\varphi(z);0)}{f(\varphi(z))}.
\end{equation*}
In particular, take such a map $\varphi$ with $\varphi(z)=\ii$ (which must exist); then we have 
\begin{equation*}
	R(\parti_2;z;0)=f(z)\times \frac{R(\parti_2;\ii;0)}{f(\ii)},
\end{equation*}
which completes the proof. 
\end{proof}

\begin{proof}[Proof of Theorem~\ref{thm::main_boundary} for $n=\ell=2$ and $\parti=(\{1,2\},\{3,4\})$]
	One can proceed as above and as in the proof of Lemma~\ref{lem::cvg_proba_aux1} for $\parti=(\{1,2\},\{3,4\})$ and in~\cite[Proof of Theorem~1.5]{Cam24} to show the existence of 
	\begin{equation}
	R(\parti;z_1,z_2;x_1,x_2):=\lim_{a\to 0}a^{-\frac{1}{4}-1} \times \mathbb{P}^a_{\mathbb{H}}\left[z_1^a\longleftrightarrow z_2^a \centernot{\longleftrightarrow}x_1^a\longleftrightarrow x_2^a\right],
	\end{equation}
with the following additional observation: we denote by $\mathcal{C}_m:=\{B_{\delta_m}(z_1)\longleftrightarrow B_{\delta_m}(z_2)\circ B_{\delta_m}(x_1)\longleftrightarrow B_{\delta_m}(x_2)\}$ the event that, if we declare closed all the edges inside $B_{\delta_m}(z_j)$, $j=1,2,3,4$, there are two disjoint open clusters connecting $B_{\delta_m}(x_1)$ to $B_{\delta_m}(x_2)$ and $B_{\delta_m}(z_1)$ to $B_{\delta_m}(z_2)$, respectively; and denote by $\mathcal{G}_{\delta_m^{\frac{99}{100}},L}(x_k)$ the event that there are three disjoint closed/open/closed arms crossing $\mathbb{H}\cap \left(B_{L}(x_k)\setminus B_{\delta_m^{\frac{99}{100}}}(x_k)\right)$; then we have (when $m$ is large enough) for some $L>0$ that is independent of $m$,  
\begin{align*}
	&\mathbb{P}_{\mathbb{H}}\left[\mathcal{C}_m\cap \{B_{\delta_k}(z_1)\longleftrightarrow B_{\delta_k}(x_1)\} | z_j\longleftrightarrow \partial B_{\epsilon}(z_j), x_r\longleftrightarrow \partial B_{\delta_k}(x_r),\enspace (j,r)\in \{1,2\}^2\right]\\
	\leq &  c \left(\frac{\sqrt{\delta_m}}{L}\right)^{\frac{35}{24}-\frac{1}{100}}\times \left(\frac{\delta_m}{\epsilon}\right)^{-1/4}+c\lim_{\eta\to 0} \frac{\mathbb{P}_{\mathbb{H}}\left[ \left(\cup_{r=1}^2 \mathcal{G}_{\delta_m^{\frac{99}{100}},L}(x_r)\right)\cap \left(\left\{ \mathcal{A}_{\eta,\delta_m}(x_r),\enspace 1\leq j\leq 2\right\}\right)\right]}{\mathbb{P}_{\mathbb{{H}}}\left[ \mathcal{A}_{\eta,\epsilon}(x_r),\enspace 1\leq j\leq 2\right]}\\
	\leq & c \left(\frac{\sqrt{\delta_m}}{L}\right)^{\frac{35}{24}-\frac{1}{100}}\times \left(\frac{\delta_m}{\epsilon}\right)^{-1/4}+c^*\mathbb{P}_{\mathbb{H}}\left[\mathcal{G}_{\delta_m^{\frac{99}{100}},L}(x_1)\right] \times \left(\frac{\delta_m}{\epsilon}\right)^{-1}\\
	\leq& c \left(\frac{\sqrt{\delta_m}}{L}\right)^{\frac{35}{24}-\frac{1}{100}}\times \left(\frac{\delta_m}{\epsilon}\right)^{-1/4}+c^{**} \left(\frac{\delta_m^{\frac{99}{100}}}{L}\right)^{\frac{5}{3}-\frac{1}{100}} \times \left(\frac{\delta_m}{\epsilon}\right)^{-1},
\end{align*}
where $c,c^*,c^{**}\in (0,\infty)$ are three constants that do not depend on $m$ and $k$. The first inequality follows from the spatial mixing property in Lemma~\ref{lem::spatial} and the proof of Lemma~\ref{lem::cvg_proba_aux1}, the second inequality uses the boundary one-arm exponent given in~\cite[Theorems~1 and 2]{PolyFKIsing} and the spatial mixing property, and the last inequality follows from the fact that $\mathbb{P}\left[\mathcal{F}_{\sqrt{\delta_m},L}(z_1)\right]\sim \left(\frac{\sqrt{\delta_m}}{L}\right)^{\frac{35}{24}+o(1)}$ as $\frac{\sqrt{\delta_m}}{L}\to 0$ and that $\mathbb{P}\left[\mathcal{G}_{\delta_m^{\frac{99}{100}},L}(x_1)\right]\sim \big(\frac{\delta_m^{\frac{99}{100}}}{L}\big)^{\frac{5}{3}+o(1)}$ as $\frac{\delta_m^{\frac{99}{100}}}{L}\to 0$, which are consequences of~\cite[Theorems~3 and 4]{PolyFKIsing} and~\cite[Theorems~1 and 2]{PolyFKIsing}, respectively.

The desired conformal covariance property of $R(\parti;z_1,z_2;x_1,x_2)$ can be derived as in the proof of part (2) of Theorem~\ref{thm::cvg_proba}, with the additional help of \eqref{eqn::scaling_bd}, which replaces Lemma~\ref{lem::scaling} for the boundary points.
\end{proof}

\appendix

\section{Exact formulas for some boundary correlation functions of the critical Ising model}

\subsection{Definitions and main results}
Suppose that $G=(V(G),E(G))$ is a finite subgraph of $\mathbb{Z}^2$.  
The \textit{Ising model} on $G$ is a random assignment $\sigma=(\sigma_v)_{v\in V(G)}\in \{\ominus,\oplus\}^{V(G)}$ of spins. The boundary condition $\tau$ is specified by three disjoint subsets $\{\oplus\}$, $\{\ominus\}$ and 
$\{\mathbf{f}\}$, which form a partition of the set of vertices in $\mathbb{Z}^2\setminus G$ that are adjacent to $\partial G$.
With boundary condition $\tau$,  and inverse-temperature $\beta>0$, the probability measure of the Ising model is given by
\[\phi_{\beta,G}^{\tau}[\sigma]=\frac{\exp\big(\beta \sum_{\langle v,w\rangle\in E(G)}\sigma_v\sigma_w+\beta\sum_{\stackrel{v\sim w}{v\in V(G),w\in\{\oplus\}}}\sigma_{v}-\beta\sum_{\stackrel{v\sim w}{v\in V(G),w\in\{\ominus\}}}\sigma_{v}\big)}{Z_{\beta,G}^{\tau}}\]
with
 \[ Z_{\beta,G}^{\tau}:=\sum_{\sigma}\exp\left(\beta \sum_{\langle v,w\rangle\in E(G)}\sigma_v\sigma_w+\beta\sum_{\stackrel{v\sim w}{v\in V(G),w\in\{\oplus\}}}\sigma_{v}-\beta\sum_{\stackrel{v\sim w}{v\in V(G),w\in\{\ominus\}}}\sigma_{v}\right). \]
In this article, we focus on the Ising model with critical inverse-temperature $\beta=\beta_c:=\frac{1}{2}\log(1+\sqrt{2})$. 

Let $x_1^a<\ldots<x_{N}^a<x_{N+1}^a<x_{N+2}^a$ be vertices in $a\mathbb{Z}$. We consider the Ising model on $a(\overline{\mathbb{H}}\cap \mathbb{Z}^2)$ with two types of boundary conditions:
\begin{itemize}
	\item free boundary condition $\{\mathbf{f}\}$, with expectation denoted by $\mathbb{E}_{\mathbb{H}}^{(a,\mathbf{f})}$;
	\item mixed free/$\oplus$ boundary condition $\{\mathbf{m}\}$:
	\begin{align} \label{eqn::Ising_bc_mixed}
		\oplus \text{ next to }[x_{N+1}^a x_{N+2}^a] ,\quad\text{and}\quad \text{free next to } \mathbb{R}\setminus [x_{N+1}^ax_{N+2}^a],
	\end{align}
where $[x_{N+1}^a x_{N+2}^a]:=[x_{N+1}^a,x_{N+2}^a]\cap  a\mathbb{Z}$; we denote by $\mathbb{E}_{\mathbb{H}}^{(a,\mathbf{m})}$ the corresponding expectation. 
\end{itemize}
Let $\#\in \{\mathbf{f},\mathbf{m}\}$. We are interested in the spin correlation $\mathbb{E}_{\mathbb{H}}^{(a,\#)}\left[\sigma_{x_1^a}\cdots \sigma_{x_N^a}\right]$. We will show that these boundary spin correlations (when normalized properly) have nontrivial conformally covariant scaling limits $\langle \sigma_{x_1}\cdots \sigma_{x_N}\rangle_{\mathbb{H}}^{\#}$, which have explicit expressions and satisfy certain BPZ equations \cite{BelavinPolyakovZamolodchikovInfiniteConformalSymmetry2D}, \cite{BelavinPolyakovZamolodchikovInfiniteConformalSymmetryCritical2D}.

We introduce some notation to present the formulas. For $n\geq 1$, we let $\Pi_n$ denote the set of all \textit{pair partitions} $\varpi=\{\{c_1,d_1\},\ldots,\{c_n,d_n\}\}$ of the set $\{1,2,\ldots,2n\}$, that is, partitions of this set into $n$ disjoint two-element subsets $\{c_j,d_j\}\subseteq \{1,2,\ldots,2n\}$, with the convention that 
\begin{align*}
	\text{$c_1<c_2<\cdots< c_N$ and $c_j<d_j$ for $j\in \{1,2,\ldots,n\}$.}
\end{align*}
We also denote by $\mathrm{sgn}(\varpi)$ the sign of the partition $\varpi$ defined as the sign of 
\begin{align*}
	\text{the product }	\prod (c-e)(c-f)(d-e)(d-f)\text{  over pairs of distinct elements $\{c,d\}, \{e,f\}\in \varpi$.}
\end{align*}

\begin{proposition}\label{prop::BPZ_free}
Suppose that $-\infty<x_1<\cdots<x_{2n}<\infty$ and let $x_1^a<\cdots<x_{2n}^a$ be $2n$ vertices in $a\mathbb{Z}$ satisfy $\lim_{a\to 0}x_j^a=x_j$ for $1\leq j\leq 2n$. Then we have, 
	\begin{equation} \label{eqn::Pfa_free}
\begin{split}
			\langle \sigma_{x_1}\cdots\sigma_{x_{2N}}\rangle_{\mathbb{H}}^{\mathbf{f}}:=&\lim_{a\to 0} a^{-2n}\times \mathbb{E}_{\mathbb{H}}^{(a,\mathbf{f})}\left[\sigma_{x_1^a}\cdots \sigma_{x_{2n}^a}\right]\\
	=&C^n_4\mathrm{Pf}\left[\frac{1}{x_k-x_j}\right]_{j,k=1}^{2n}=C^n_4\sum_{\varpi\in\Pi_n}\mathrm{sgn}(\varpi)\prod_{\{c,d\}\in \varpi} \frac{1}{x_d-x_c}, 
\end{split}
	\end{equation}
	where $C_4$ is the constant in~\eqref{eqn::FK_bd_two_three}. As a consequence, for each $j\in \{1,2,\ldots,2n\}$, the function $\langle \sigma_{x_1}\ldots \sigma_{x_{2N}}\rangle_{\mathbb{H}}^{\mathbf{f}}$ defined by~\eqref{eqn::Pfa_free} is annihilated by
	\begin{equation} \label{eqn::BPZ_free}
	\frac{3}{2}\partial^2_{j}+\sum_{k\neq j} \frac{2}{x_k-x_j}\partial_k-\frac{1}{(x_k-x_j)^2}.
\end{equation}
\end{proposition}
\begin{proof}
	It is well-known that $\mathbb{E}_{\mathbb{H}}^{(a,\mathbf{f})}\left[\sigma_{x_1^a}\cdots\sigma_{x_{2N}^a}\right]$ has the following Pfaffian expression\footnote{The Pfaffian relation \eqref{eqn::Paffian_free} is valid for all $\beta\geq 0$.}~\cite{Paffian} (see also \cite[Section~1.4]{EmergentPlanarity} for a new proof):
	\begin{equation} \label{eqn::Paffian_free}
\mathbb{E}_{\mathbb{H}}^{(a,\mathbf{f})}\left[\sigma_{x_1^a}\cdots\sigma_{x_{2n}^a}\right]=\mathrm{Pf}\left[\mathbb{E}_{\mathbb{H}}^{(a,\mathbf{f})}\left[\sigma_{x_j^a}\sigma_{x_k^a}\right]\right]_{j,k=1}^{2n}=\sum_{\varpi\in\Pi_{n}}\mathrm{sgn}(\varpi) \sum_{\{c,d\}\in \varpi} \mathbb{E}_{\mathbb{H}}^{(a,\mathbf{f})}\left[\sigma_{x_c^a}\sigma_{x_d^a}\right].
	\end{equation}
	Combining Theorem~\ref{thm::main_boundary},~\eqref{eqn::Paffian_free} with Edwards-Sokal coupling (see~\cite{EdwardsSokal}), we obtain~\eqref{eqn::Pfa_free}. Combining~\eqref{eqn::Pfa_free} with~\cite[Proposition~4.6]{KytolaPeltolaPurePartitionFunctions}, we obtain~\eqref{eqn::BPZ_free}. 
\end{proof}

The situation for the mixed boundary condition~\eqref{eqn::Ising_bc_mixed} is more complicated, even though one still has the Pfaffian structure for the boundary spin correlations. Indeed, already for $N=2$, the two-point spin correlation $\langle \sigma_{x_1}\sigma_{x_2}\rangle_{\mathbb{H}}^{\mathbf{m}}$ in the continuum is a conformally covariant function of four variables, $x_1,x_2,x_3$ and $x_4$, whose functional form, however, is not fully determined by its conformal covariance property. Instead, we will figure out its expression by relating it to the $\mathrm{SLE}_3$ partition function via the high-temperature expansion of the Ising model (see Lemmas~\ref{lem::high_tempe} and~\ref{lem::low_tempe} below and~\cite[Theorem~3.1]{IzyurovObservableFree}). 
\begin{theorem} \label{thm::correlation_one_plus}
	Suppose that $-\infty<x_1<\ldots <x_{N}<x_{N+1}<x_{N+2}<\infty$ and let $x_1^a<\ldots<x_{N}^a<x_{N+1}^a<x_{N+2}^a$ be $N+2$ vertices on $a\mathbb{Z}$ which satisfy $\lim_{a\to 0}x_j^a=x_{j}$ for $1\leq j\leq N+2$. Then there exist constants $C_9,C_{10}\in (0,\infty)$ such that 
	\begin{equation} \label{eqn::scaling_corr_one_plus}
\begin{split}
			\langle \sigma_{x_1}\ldots\sigma_{x_N}\rangle_{\mathbb{H}}^{\mathbf{m}}:=&\lim_{a\to 0} a^{-\frac{N}{8}} \times \mathbb{E}^{(a,\mathbf{m})}_{\mathbb{H}}\left[\sigma_{x_1^a}\cdots\sigma_{x_N^a}\right]\\
			=&\begin{cases}
		C_9^n \mathcal{R}_{2n}(x_1,\ldots,x_{2n};x_{2n+1},x_{2n+2}) &\text{if }N=2n,\\
		C_{10} C_9^n \mathcal{R}_{2n+1}(x_1,\ldots,x_{2n+1};x_{2n+2},x_{2n+3}), &\text{if }N=2n+1,
	\end{cases}
\end{split}
	\end{equation}
	where $\mathcal{R}_{2n}$ and $\mathcal{R}_{2n+1}$ are defined by~\eqref{eqn::def_Ising_total_par_even} and~\eqref{eqn::def_Ising_total_par_odd} below. Moreover, for $j\in \{1,2,\ldots,N\}$, the function $\langle \sigma_{x_1}\ldots\sigma_{x_{N}}\rangle_{\mathbb{H}}^{\mathbf{m}}$ defined by~\eqref{eqn::scaling_corr_one_plus} is annihilated by the differential operator
	\begin{equation}\label{eqn::BPZ_one_plus}
			\frac{3}{2}\partial^2_{j}+\sum_{k\neq j} \frac{2}{x_k-x_j}\partial_j-\frac{2\Delta_k}{(x_k-x_j)^2},
	\end{equation}
	where $\Delta_1=\Delta_2=\ldots=\Delta_N=\frac{1}{2}$ and $\Delta_{N+1}=\Delta_{N+2}=0$. 
\end{theorem}

Now, let us define the functions $\mathcal{R}_N$ in Theorem~\ref{thm::correlation_one_plus}. For $m\geq 1$, we write
\begin{equation*}
	\chamber_{m}:=\{(x_1,\ldots,x_m)\in \mathbb{R}^m: x_1<x_2<\ldots<x_m\}.
\end{equation*}
When $N=2n$ with $n\ge 1$, we define $\mathcal{R}_N: \chamber_{N+2}\to \R$ by
\begin{equation} \label{eqn::def_Ising_total_par_even}
	\begin{split}
		&	\mathcal{R}_{2n}(x_1,\ldots,x_{2n};x_{2n+1},x_{2n+2})=\prod_{k=1}^{2n}\frac{1}{\sqrt{(x_{2n+2}-x_k)(x_{2n+1}-x_k)}}\\
	&	\qquad\qquad\qquad\qquad\qquad\qquad\times \sum_{\varpi\in\Pi_n}\mathrm{sgn}(\varpi)\prod_{\{c,d\}\in\varpi}\frac{(x_{2n+1}-x_c)(x_{2n+2}-x_d)+(x_{2n+1}-x_d)(x_{2n+2}-x_c)}{x_d-x_c}. 
	\end{split}
\end{equation}
When $N=2n+1$ with $n\ge 0$, we define $\mathcal{R}: \chamber_{N+2}\to \R$ by
\begin{equation} \label{eqn::def_Ising_total_par_odd}
	\begin{split}
		&	\mathcal{R}_{2n+1}(x_1,\ldots,x_{2n+1};x_{2n+2},x_{2n+3})=(x_{2n+3}-x_{2n+2})^{\frac{1}{2}}\times \prod_{k=1}^{2n+1}\frac{1}{\sqrt{(x_{2n+2}-x_k)(x_{2n+3}-x_k)}}\\
		&\qquad\qquad\qquad\qquad\qquad\quad\times \sum_{\varpi\in\Pi_{n+1}}\mathrm{sgn}(\varpi)\prod_{\genfrac{}{}{0pt}{}{\{c,d\}\in \varpi}{d\neq 2n+2}} \frac{(x_{2n+2}-x_c)(x_{2n+3}-x_d)+(x_{2n+2}-x_d)(x_{2n+3}-x_c)}{x_d-x_c}.
	\end{split}
\end{equation}

\begin{remark}
	We emphasize that our arguments allow one to extend the boundary condition~\eqref{eqn::Ising_bc_mixed} to more general alternating $\mathrm{free}/\mathrm{wired}$ boundary conditions, where a ``$\mathrm{wired}$''  boundary segment means that the spins on this segment are conditioned to be the same. 
\end{remark}


We now proceed with the proof of~Theorem~\ref{thm::correlation_one_plus}.
\subsection{Proof of Theorem~\ref{thm::correlation_one_plus} modulo a key lemma}
We start by showing that, with mixed boundary conditions~\eqref{eqn::Ising_bc_mixed}, Ising boundary spin correlations have a Pfaffian structure analogous to \eqref{eqn::Paffian_free}, which is valid for free boundary conditions.
\begin{lemma} \label{lem::Paffian_mixed}
	Let $x_1^a<x_2^a<\ldots<x_{N}^a<x_{N+1}^a<x_{N+2}^a$ be vertices in $a\mathbb{Z}\cap \mathbb{R}$. Consider Ising model\footnote{The results in Lemma~\ref{lem::Paffian_mixed} hold for generic inverse temperature $\beta\geq 0$.} on $a(\overline{\mathbb{H}}\cap \mathbb{Z}^2)$ with the mixed boundary condition~\eqref{eqn::Ising_bc_mixed}. If $N=2n$, then we have 
	\begin{align*}
		\mathbb{E}_{\mathbb{H}}^{(a,\mathbf{m})}\left[\sigma_{x_{1}^a}\cdots \sigma_{x_{2n}^a}\right]=\sum_{\varpi\in\Pi_n}\mathrm{sgn}(\varpi)\prod_{\{c,d\}\in \varpi}  \mathbb{E}_{\mathbb{H}}^{(a,\mathbf{m})}\left[\sigma_{x_c^a}\sigma_{x_d^a}\right].
	\end{align*}
	If $N=2n+1$, then we have 
	\begin{align*}
		\mathbb{E}_{\mathbb{H}}^{(a,\mathbf{m})}[\sigma_{x_1^a}\cdots \sigma_{x_{2n+1}^a}]=\sum_{\varpi\in\Pi_{n+1}}\mathrm{sgn}(\varpi) \mathbb{E}_{\mathbb{H}}^{(a,\mathbf{m})}[\sigma_{x_{c'}^a}] \times\prod_{\{c,d\}\in \varpi\atop 
		d\neq 2n+2} \mathbb{E}_{\mathbb{H}}^{(a,\mathbf{m})}\left[\sigma_{x_c^a}\sigma_{x_d^a}\right],
	\end{align*}
	where $c'$ denotes the index paired to $2n+2$ in $\varpi$. 
\end{lemma}
\begin{proof}
One can basically mimic the proof of the Pfaffian structure of the boundary spin correlations for the free boundary condition in~\cite[Section~1.4]{EmergentPlanarity}. Alternatively, one can use the same trick as in~\eqref{eqn::high_tempe_one_two_aux0} below to express the spin correlations for the mixed boundary condition as the limit of a sequence of spin correlations for free boundary conditions and then utilize the known Pfaffian structure for the latter.
\end{proof}

\begin{lemma} \label{lem::correlation_one_plus_aux1}
With the notation of Theorem~\ref{thm::correlation_one_plus}, suppose that $N=1$, then there exists a constant $C_{10}\in (0,\infty)$ such that 
	\begin{align}\label{eqn::correlation_one_plus_aux1}
		\lim_{a\to 0} a^{-\frac{1}{2}}\times \mathbb{E}_{\mathbb{H}}^{(a,\mathbf{m})}\left[\sigma_{x_1^a}\right]=C_{10}\frac{\sqrt{x_3-x_2}}{\sqrt{x_3-x_1}\sqrt{x_2-x_1}}. 
	\end{align}
\end{lemma}
\begin{proof}
	One can proceed as in the proof of Theorem~\ref{thm::main_boundary} to show that 
	\begin{align*}
	f_{\mathbb{H}}(x_1;x_2,x_3):=	\lim_{a\to 0} a^{-\frac{1}{2}}\times \mathbb{E}_{\mathbb{H}}^{(a,\mathbf{m})}\left[\sigma_{x_1^a}\right]\in (0,\infty);
	\end{align*}
	moreover, for any M\"obius map $\varphi$ of the upper half-plane with $\varphi(x_j)\neq \infty$ for $1\leq j\leq 3$, we have 
	\begin{align*}
		f_{\mathbb{H}}(\varphi(x_1);\varphi(x_2),\varphi(x_3))=\vert\varphi'(x_1)\vert^{-\frac{1}{2}}\times f_{\mathbb{H}}(x_1;x_2,x_3). 
	\end{align*}
	This M\"obius covariance of $f_{\mathbb{H}}$ implies that there exists a constant $C_{10}\in (0,\infty)$ such that~\eqref{eqn::correlation_one_plus_aux1} holds. 
\end{proof}

\begin{lemma}\label{lem::correlation_one_plus_aux2}
	With the notation of Theorem~\ref{thm::correlation_one_plus}, suppose that $N=2$, then there exists a constant $C_9\in (0,\infty)$ such that 
	\begin{align}  \label{eqn::correlation_one_plus_aux2}
	\lim_{a\to 0} a^{-1}\times \mathbb{E}_{\mathbb{H}}^{(a,\mathbf{m})}\left[\sigma_{x_1^a}\sigma_{x_2^a}\right]=C_9\frac{(x_4-x_1)(x_3-x_2)+(x_4-x_2)(x_3-x_1)}{(x_2-x_1)\sqrt{x_3-x_1}\sqrt{x_4-x_1}\sqrt{x_3-x_2}\sqrt{x_4-x_2}}.
	\end{align}
\end{lemma}
We note that the expression on the right-hand side of~\eqref{eqn::correlation_one_plus_aux2} is the partition function of some $\mathrm{SLE}_3$ variant (see~\cite[Section~3]{IzyurovObservableFree}).
The proof of Lemma~\ref{lem::correlation_one_plus_aux2} is more involved and we postpone it to the next section.

\begin{proof}[Proof of Theorem~\ref{thm::correlation_one_plus}]
The relation~\eqref{eqn::scaling_corr_one_plus} follows directly from Lemmas~\ref{lem::Paffian_mixed}-\ref{lem::correlation_one_plus_aux2}. 

It remains to show that the function $\LR_N$ defined by~\eqref{eqn::def_Ising_total_par_even}-\eqref{eqn::def_Ising_total_par_odd} satisfies the PDEs~\eqref{eqn::BPZ_one_plus}. Indeed, as a special case of~\cite[Theorem~3.1]{IzyurovObservableFree}, the function $\mathcal{R}_N$ is the partition function of certain local multiple $\mathrm{SLE}_3$ paths. Then, the PDEs~\eqref{eqn::BPZ_one_plus} follow from the commutation relations~\cite[Theorem~7]{DubedatCommutationSLE}, see also~\cite[Appendix~A]{KytolaPeltolaPurePartitionFunctions}. 
\end{proof}

\subsection{Proof of Lemma~\ref{lem::correlation_one_plus_aux2}}
With the notation of Theorem~\ref{thm::correlation_one_plus}, suppose that $N=2$. One can proceed as in the proof of Theorem~\ref{thm::main_boundary} to show that 
\begin{align*}
	f_{\mathbb{H}}(x_1,x_2;x_3,x_4):=\lim_{a\to 0}a^{-1}\times\mathbb{E}_{\mathbb{H}}^{(a,\mathbf{m})}\left[\sigma_{x_1^a}\sigma_{x_2^a}\right]\in(0,\infty); 
\end{align*}
moreover, for any M\"obius map $\varphi$ of the upper half-plane with $\varphi(x_j)\neq \infty$ for $1\leq j\leq 4$, we have 
\begin{align*}
	f_{\mathbb{H}}(\varphi(x_1),\varphi(x_2);\varphi(x_3),\varphi(x_4))=\vert\varphi'(x_1)\vert^{-\frac{1}{2}} \vert\varphi'(x_2)\vert^{-\frac{1}{2}}\times f_{\mathbb{H}}(x_1,x_2;x_3,x_4).
\end{align*}
However, this M\"obius covariance property is not sufficient to specify the functional form of $f_{\mathbb{H}}(x_1,x_2;x_3,x_4)$. Instead, we adopt the following strategy:
\begin{itemize}
	\item First, at the critical point, using the high-temperature expansion, we relate the correlation $\mathbb{E}_{\mathbb{H}}^{(a,\mathbf{m})}[\sigma_{x_1^a}\sigma_{x_2^a}]$ to the low-temperature expansion of the Ising model on the dual graph;
	\item Second, using the integrability result of Smirnov's Ising fermionic observable for free boundary conditions studied in~\cite{IzyurovObservableFree}, we figure out the scaling limit of the low-temperature expansion of the Ising model on the dual graph in the first step. 
\end{itemize}
To this end, we need to consider the Ising model on a finite domain first. 

Let $M>0$ and let $\varphi$ be a conformal map from $\mathbb{H}$ onto $[-M,M]\times [0,M]$ with $-M<\varphi(x_1)<\varphi(x_2)<\varphi(x_3)<\varphi(x_4)<M$. Write $y_j=\varphi(x_j)$ for $1\leq j\leq 4$.  Define $\Omega_M:=[-M,M]\times [0,M]$ and $\Omega^a_M:=\Omega_M\cap a\mathbb{Z^2}$. Let $y_j^a\in \partial \Omega^a_M\cap \mathbb{R}$ satisfy $\lim_{a\to 0}y_j^a=y_j$ for $1\leq j\leq 4$. We consider the critical Ising model on $\Omega^a_M$ with the following mixed boundary conditions:
\begin{align*}
	\oplus\text{ next to }[y_3^ay_4^a],\quad \text{and}\quad \text{free next to }\partial\Omega_M^a\setminus [y_3^ay_4^a],
\end{align*}
and we denote by $\mathbb{E}_{\Omega_M}^{(a,\mathbf{f})}$ the corresponding expectation. 

We now introduce some notation that will be used to define the high-temperature expansion and the Ising fermionic observable for free boundary conditions initially introduced in~\cite[Section~2]{IzyurovObservableFree}. 
Define $\overline{\Omega}_{M}^a$ to be the graph whose vertex set $V(\overline{\Omega}_M^a)$ equals 
\[ 
V(\Omega_{M}^a)\cup \big([y_3^ay_4^a]-\ii a\big), \; \text{ where } \big([y_3^ay_4^a]-\ii a\big) := \{w\notin V(\Omega_M^a): \exists v\in [y_3^ay_4^a]\text{ such that }v\sim w\}, \]
and whose edge set consists of edges in $a\mathbb{Z}^2$ connecting vertices in $V(\overline{\Omega}_M^a)$. 

For each vertex $v$ of $\overline{\Omega}_M^{a}$, we add four vertices $c_j$ at $v+\frac{\sqrt{2}a}{4}\exp(\frac{\ii \pi}{4}+\frac{\ii \pi}{2}j)$, $j=0,1,2,3$, and connect each $c_j$ by an edge to $v$; the four vertices are called \textit{corners} and the corresponding edges are called \textit{corner edges}. We add a vertex to the midpoint of each edge on $a\mathbb{Z}^2$. We will often identify a corner edge with the corresponding corner, and identify an edge of $\overline{\Omega}_M^a$ with its midpoint. By a \textit{discrete outer normal} at a vertex $v\in \partial \overline{\Omega}_M^a$, we mean an oriented edge connecting $v$ to a corner or to a midpoint adjacent to $v$ but not in $\overline{\Omega}^a_M$, pointing away from $v$. We will often identify a discrete outer normal with the corresponding corner or midpoint. Denote by $\tilde{V}(\overline{\Omega}^a_M)$  the set of vertices in $\overline{\Omega}^a_M$, together with the midpoints and corners adjacent to $\overline{\Omega}^a_M$. Denote by $\tilde{E}(\overline{\Omega}^a_M)$ the set of primal edges, half-edges, corners, and discrete outer normals of $\overline{\Omega}^a_M$. Define the weights $\mathrm{w}_e$  for $e\in \tilde{E}(\overline{\Omega}^a_M)$ by
\[\mathrm{w}_e:=\begin{cases}
	\sqrt{2}-1, & \text{if $e$ is an edge in $a\mathbb{Z}^2$};\\
	(\sqrt{2}-1)^{\frac{1}{2}},  & \text{if $e$ is a half-dege};\\
	(\sqrt{2}-1)^{\frac{1}{2}}\cos(\frac{\pi}{8}),  & \text{if $e$ is a corner edge.}
\end{cases}\] 
For $m\geq 0$ and distinct elements $z_1,\ldots,z_{2m}\in \tilde{V}(\overline{\Omega}_M^a)$, denote by $\text{Conf}(\overline{\Omega}_M^a;\{z_1,z_2,\ldots,z_{2m}\})$  the set of all subsets $S$ of $\tilde{E}(\overline{\Omega}_M^a)$ such that all generalized vertices in $\tilde{V}(\overline{\Omega}_M^a)$, except for $z_1,\ldots,z_{2m}$,  have an even degree in $S$, and write
\begin{align*}
	Z(\overline{\Omega}_M^a;\{z_1,z_2,\ldots,z_m\}):=\sum_{S\in \text{Conf}(\overline{\Omega}_M^a;\{z_1,z_2,\ldots,z_{2m}\})} \prod_{e\in S\setminus \big([y_3^ay_4^a]-\ii a\big)} \mathrm{w}_e. 
\end{align*}

\paragraph*{High-temperature expansion for the mixed boundary condition.} 
Let $A\subseteq V(\Omega_M^a)$ with cardinality $\#A\geq 1$ and write $\sigma_A:=\prod_{v\in A}\sigma_v$. It follows from our definitions that 
\begin{align} \label{eqn::high_tempe_one_two_aux}
	\mathbb{E}_{\Omega_M}^{(a,\mathbf{m})} [\sigma_A]=\frac{\sum_{\sigma\in \{\pm 1\}^{V(\Omega_M^a)}}\sigma_A\exp\big(\beta\sum_{\langle v,w\rangle \in E(\Omega_M^a)}\sigma_v\sigma_w+\beta\sum_{v\sim w\atop v\in \Omega_M^a,w\in ([y^a_3,y^a_4]-\ii a)}\sigma_v\big)}{\sum_{\sigma\in \{\pm1\}^{V(\Omega_M^a)}}\exp\Big(\beta\sum_{\langle v,w\rangle \in E(\Omega_M^a)}\sigma_v\sigma_w+\beta\sum_{v\sim w\atop v\in \Omega_M^a,w\in ([y^a_3,y^a_4]-\ii a)}\sigma_v\Big)}.
\end{align}

\begin{lemma} \label{lem::high_tempe}
	Let $A\subseteq V(\Omega_M^a)$, then we have 
	\begin{align}\label{eqn::high_tempe}
\mathbb{E}_{\Omega_{M}}^{(a,\mathbf{m})}[\sigma_A]= \begin{cases}
	Z(\overline{\Omega}_M^a;A)/Z(\overline{\Omega}_M^a),& \text{if }\#A \text{ is even},\\
	Z(\overline{\Omega}_M^a;A\cup \{y_3^a-\ii a\})/Z(\overline{\Omega}_M^a),& \text{if }\#A \text{ is odd}.
\end{cases}
	\end{align}
In particular, we have 
	\begin{align*}
				\mathbb{E}_{\Omega_M}^{(a,\mathbf{m})} [\sigma_{y_1^a}]= \frac{Z(\overline{\Omega}_M^a;\{y_1^a,y_3^a-\ii a\})}{Z(\overline{\Omega}_M^a)},\quad
				\mathbb{E}_{\Omega_M^a}^{(a,\mathbf{m})} [\sigma_{y_1^a}\sigma_{y_2^a}]= \frac{Z(\overline{\Omega}_M^a;\{y_1^a,y_2^a\})}{Z(\overline{\Omega}_M^a)}.
	\end{align*}
\end{lemma}
\begin{proof}
Throughout the proof, we let $\sigma_{y_3^a-\ii a}=1$. Now we express~\eqref{eqn::high_tempe_one_two_aux} in a different way:
\begin{align} \label{eqn::high_tempe_one_two_aux0}
	\mathbb{E}_{\Omega_{M}}^{(a,\mathbf{m})}[\sigma_A]=\lim_{\beta\to +\infty}\frac{\sum_{\sigma\in \{\pm 1\}^{V(\overline{\Omega}_M^a\setminus \{y_3^a-\ii a\})}}\sigma_A \exp \big(\beta_c \sum_{\langle v,w\rangle \in E(\overline{\Omega}_M^a)\setminus ([y_3^ay_4^a]-\ii a )}\sigma_v\sigma_w+\beta \sum_{\langle v,w\rangle\in ([y_3^ay_4^a]-\ii a)}\sigma_{v}\sigma_w\big)}{\sum_{\sigma\in \{\pm 1\}^{V(\overline{\Omega}_M^a\setminus \{y_3^a-\ii a\})}}\exp \big(\beta_c \sum_{\langle v,w\rangle \in E(\overline{\Omega}_M^a)\setminus ([y_3^ay_4^a]-\ii a )}\sigma_v\sigma_w+\beta \sum_{\langle v,w\rangle\in ([y_3^ay_4^a]-\ii a)}\sigma_{v}\sigma_w\big)}.
\end{align} 

Note that for $\sigma_v,\sigma_w\in \{\pm 1\}$, we have
\begin{align}\label{eqn::high_tempe_aux1}
	\exp(\beta\sigma_v\sigma_w)= \cosh(\beta)\left[1+\tanh(\beta)\sigma_v\sigma_w\right],\quad \tanh(\beta_c)=\sqrt{2}-1,\quad \lim_{\beta\to +\infty}\tanh(\beta)=1. 
\end{align}
As a consequence of~\eqref{eqn::high_tempe_aux1}, we have
\begin{align} \label{eqn::high_tempe_one_two_aux1}
\begin{split}
		&\sum_{\sigma\in \{\pm 1\}^{V(\overline{\Omega}_M^a\setminus \{y_3^a-\ii a\})}}\sigma_A \exp \big(\beta_c \sum_{\langle v,w\rangle \in E(\overline{\Omega}_M^a)\setminus ([y_3^ay_4^a]-\ii a )}\sigma_v\sigma_w+\beta \sum_{\langle v,w\rangle\in ([y_3^ay_4^a]-\ii a)}\sigma_{v}\sigma_w\big)\\
	&\quad= \sum_{S\subseteq E(\overline{\Omega}_M^a)} (\sqrt{2}-1)^{\# S\setminus ([y_3^ay_4^a]-\ii a)} \tanh(\beta)^{\# S\cap ([y_3^ay_4^a]-\ii a)}\sum_{\sigma\in \{\pm1\}^{V(\overline{\Omega}_M^a)\setminus \{y_3^a-\ii a\}}}\sigma_A\prod_{\langle v,w\rangle}\sigma_v\sigma_w\\
	&\quad\quad\quad\times \cosh(\beta_c)^{\# E(\overline{\Omega}_M^a)\setminus ([y_3^ay_4^a]-\ii a)} \cosh(\beta)^{\#([y_3^ay_4^a]-\ii a)}. 
\end{split}
\end{align}
Note that, if $\# A$ is even, then
\begin{align*}
	\sum_{\sigma\in \{\pm1\}^{V(\overline{\Omega}_M^a\setminus \{y_3^a-\ii a\})}}\sigma_A\prod_{\langle v,w\rangle}\sigma_v\sigma_w=\begin{cases}
		2^{\# V(\overline{\Omega}_M^a\setminus \{y_3^a-\ii a\})},&\text{if }S\in \mathrm{Conf}(\overline{\Omega}_M^a; A),\\
		0,& \text{otherwise};
	\end{cases}
\end{align*}
if $\# A$ is odd, then 
\begin{align*}
	\sum_{\sigma\in \{\pm1\}^{V(\overline{\Omega}_M^a\setminus \{y_3^a-\ii a\})}}\sigma_A\prod_{\langle v,w\rangle}\sigma_v\sigma_w=\begin{cases}
		2^{\# V(\overline{\Omega}_M^a\setminus \{y_3^a-\ii a\})},&\text{if }S\in \mathrm{Conf}(\overline{\Omega}_M^a; A,y_3^a-\ii a),\\
		0,& \text{otherwise}.
	\end{cases}
\end{align*}
Plugging these two observations into~\eqref{eqn::high_tempe_one_two_aux1} shows that, if $\# A$ is even, then 
\begin{align} \label{eqn::high_tempe_one_two_aux2}
	\begin{split}
		&\sum_{\sigma\in \{\pm 1\}^{V(\overline{\Omega}_M^a\setminus \{y_3^a-\ii a\})}}\sigma_A \exp \big(\beta_c \sum_{\langle v,w\rangle \in E(\overline{\Omega}_M^a)\setminus ([y_3^ay_4^a]-\ii a )}\sigma_v\sigma_w+\beta \sum_{\langle v,w\rangle\in ([y_3^ay_4^a]-\ii a)}\sigma_{v}\sigma_w\big)\\
		&\quad=\cosh(\beta_c)^{\# E(\overline{\Omega}_M^a)\setminus ([y_3^ay_4^a]-\ii a)} \cosh(\beta)^{\#([y_3^ay_4^a]-\ii a)} 2^{\# V(\overline{\Omega}_M^a\setminus \{y_3^a-\ii a\})}\\
		&\quad\quad \times \sum_{S\in \mathrm{Conf}(\overline{\Omega}_M^a;A)} (\sqrt{2}-1)^{\# S\setminus ([y_3^ay_4^a]-\ii a)} \tanh(\beta)^{\# S\cap([y_3^ay_4^a]-\ii a)};
	\end{split}
\end{align}
if $\# A$ is odd, then 
\begin{align} \label{eqn::high_tempe_one_two_aux3}
	\begin{split}
		&\sum_{\sigma\in \{\pm 1\}^{V(\overline{\Omega}_M^a\setminus \{y_3^a-\ii a\})}}\sigma_A \exp \big(\beta_c \sum_{\langle v,w\rangle \in E(\overline{\Omega}_M^a)\setminus ([y_3^ay_4^a]-\ii a )}\sigma_v\sigma_w+\beta \sum_{\langle v,w\rangle\in ([y_3^ay_4^a]-\ii a)}\sigma_{v}\sigma_w\big)\\
		&\quad=\cosh(\beta_c)^{\# E(\overline{\Omega}_M^a)\setminus ([y_3^ay_4^a]-\ii a)} \cosh(\beta)^{\#([y_3^ay_4^a]-\ii a)} 2^{\# V(\overline{\Omega}_M^a\setminus \{y_3^a-\ii a\})}\\
		&\quad\quad \times \sum_{S\in \mathrm{Conf}(\overline{\Omega}_M^a;A,y_3^a-\ii a)} (\sqrt{2}-1)^{\# S\setminus ([y_3^ay_4^a]-\ii a)} \tanh(\beta)^{\# S\cap([y_3^ay_4^a]-\ii a)}.
	\end{split}
\end{align}
Plugging~\eqref{eqn::high_tempe_one_two_aux2} and~\eqref{eqn::high_tempe_one_two_aux3} into~\eqref{eqn::high_tempe_one_two_aux0} gives~\eqref{eqn::high_tempe} and completes the proof. 
\end{proof}

Recall that $\varphi$ is a conformal map from $\mathbb{H}$ onto $[-M,M]\times [0,M]$ with $y_j=\varphi(x_j)$ for $1\leq j\leq 4$. 
\begin{lemma} \label{lem::low_tempe}
There exists a constant $C_{11}\in (0,\infty)$ such that
	\begin{align} \label{eqn::low_tempe}
		\lim_{a\to 0} a^{-\frac{1}{2}} \times \frac{Z(\overline{\Omega}_M^a;\{y_1^a,y_2^a\})}{Z(\overline{\Omega}_M^a;\{y_1^a,y_3^a-\ii a\})}= C_{11} \vert\varphi'(x_2) \vert^{-\frac{1}{2}}\times\frac{(x_4-x_1)(x_3-x_2)+(x_4-x_2)(x_3-x_1)}{(x_2-x_1)\sqrt{x_3-x_2}\sqrt{x_4-x_2}\sqrt{x_4-x_3}}.
	\end{align}
\end{lemma}
We postpone the proof of Lemma~\ref{lem::low_tempe} to the end of this section. With Lemmas~\ref{lem::correlation_one_plus_aux1},~\ref{lem::high_tempe} and~\ref{lem::low_tempe} at hand, we are ready to prove Lemma~\ref{lem::correlation_one_plus_aux2}.

\begin{proof}[Proof of Lemma~\ref{lem::correlation_one_plus_aux2}]
On the one hand, one can proceed as in the proof of Theorem~\ref{thm::main_boundary} to show that 
\begin{align}
	\lim_{a\to 0}a^{-\frac{1}{2}}\times \mathbb{E}_{\Omega_M}^{(a,\mathbf{m})}[\sigma_{y_1^a}]=&\vert \varphi'(x_1)\vert^{-\frac{1}{2}} \times \lim_{a\to 0} a^{-\frac{1}{2}}\times \mathbb{E}_{\mathbb{H}}^{(a,\mathbf{m})} [\sigma_{x_1^a}], \label{eqn::cov_one_point}\\
	\lim_{a\to 0}a^{-1}\times \mathbb{E}_{\Omega_M}^{(a,\mathbf{m})}[\sigma_{y_1^a}\sigma_{y_2^a}]=&\vert \varphi'(x_1)\vert^{-\frac{1}{2}} \vert \varphi'(x_2)\vert^{-\frac{1}{2}} \times \lim_{a\to 0} a^{-1}\times \mathbb{E}_{\mathbb{H}}^{(a,\mathbf{m})} [\sigma_{x_1^a}\sigma_{x_2^a}].  \label{eqn::cov_two_point}
\end{align}	
	
On the other hand, thanks to Lemma~\ref{lem::high_tempe}, we can write
\begin{align*}
a^{-1}\times \mathbb{E}_{\Omega_M}^{(a,\mathbf{m})}[\sigma_{y_1^a}\sigma_{y_2^a}]= a^{-\frac{1}{2}}\times \mathbb{E}_{\Omega_M^a}^{(a,\mathbf{m})}[\sigma_{y_1^a}]\times a^{-\frac{1}{2}} \times \frac{Z(\overline{\Omega}_M^a;\{y_1^a,y_2^a\})}{Z(\overline{\Omega}_M^a;\{y_1^a,y_3^a-\ii a\})}.
\end{align*}
According to Lemma~\ref{lem::correlation_one_plus_aux1}, we have 
\begin{align*}
\lim_{a\to 0}a^{-\frac{1}{2}}\times \mathbb{E}_{\mathbb{H}}^{(a,\mathbf{m})}[\sigma_{x_1^a}]= C_{10} \frac{\sqrt{x_4-x_3}}{\sqrt{x_4-x_1}\sqrt{x_3-x_1}},
\end{align*}
where $C_{10}$ is the constant in Lemma~\ref{lem::correlation_one_plus_aux1}.
Combining these with~\eqref{eqn::cov_one_point}, Lemma~\ref{lem::low_tempe} gives 
\begin{align} \label{eqn::two_point_finite_domain}
	\lim_{a\to 0} a^{-1}\times \mathbb{E}_{\Omega_M}^{(a,\mathbf{m})}[\sigma_{y_1^a}\sigma_{y_2^a}]=C_{10}C_{11} \vert\varphi'(x_1)\vert^{-\frac{1}{2}}\vert\varphi'(x_2)\vert^{-\frac{1}{2}}\times \frac{(x_4-x_1)(x_3-x_2)+(x_4-x_2)(x_3-x_1)}{(x_2-x_1)\sqrt{x_3-x_1}\sqrt{x_4-x_1}\sqrt{x_3-x_2}\sqrt{x_4-x_2}},
\end{align}
where $C_{11}$ is the constant in Lemma~\ref{lem::low_tempe}. 
Combining~\eqref{eqn::cov_two_point} with~\eqref{eqn::two_point_finite_domain} gives~\eqref{eqn::correlation_one_plus_aux2} with $C_9=C_{10}C_{11}$. This completes the proof. 
\end{proof}

The remaining goal is to prove Lemma~\ref{lem::low_tempe}. 
\paragraph*{Ising fermionic observable for free boundary conditions} We will use the observable initially introduced in~\cite[Section~2]{IzyurovObservableFree}. We briefly recall its construction in our setup. 

We denote by $b_1^a$ the discrete outer normal pointing from $y_1^a$ to $y_1^a-\ii a$, by $b_2^a$ the discrete outer normal pointing from $y_2^a$ to $y_2^a-\ii a$, and by $b_3^a$ the corner edge pointing from $y_3^a-\ii a$ to $y_3^a-\ii a-\frac{\sqrt{2}a}{4}\exp(\frac{\ii \pi}{4})$. 
 For each oriented edge $e$, view it as a complex number, and associate another number $\varkappa(e)\in \mathbb{C}$ to it defined by
\[\varkappa(e):=\left(\frac{\ii e}{|e|}\right)^{-1/2},\]
where $e$ is interpreted as a complex number.
Note that $\varkappa(e)$ is defined up to a sign.
We define $F^{a}$ on $\tilde{V}(\overline{\Omega}_M^a)\setminus\{y_1^a-\ii a\}$, except for the midpoints on $([y_3^ay_4^a]-\ii a)$, as
\begin{equation}
	F^{a}(z):=
	\ii \varkappa(b_1^a)\frac{\sum_{S\in \text{Conf}(\overline{\Omega}_M^a;b_1^{\delta},z)}\left(\prod_{e\in S\setminus \big([y_3^ay_4^a]-\ii a\big)}w_e\right)\exp(-\ii W(S)/2)}{(\sqrt{2}-1)^{\frac{3}{2}}\cos \frac{\pi}{8}\times Z(\overline{\Omega}_M^a;\{y_1^a,y_3^a-\ii a\})}, 
\end{equation}
where $W(S)$ is defined as follows: $S$  can be decomposed into a union of loops and a path $\gamma$ from $y_1^a-\ii a$ to $z$ in such a way that no edge is traced twice, and the loops and $\gamma$ do not cross each other or themselves transversally; the number $W(S)$ is defined to be the winding of the path $\gamma$; the winding factor $\exp(\ii W(S)/2)$ does not depend on the decomposition of $S$. Note that $F^{a}$ is only defined up to a sign.

Define
\begin{align} \label{eqn::def_obser_free_con_1}
	F(z;\mathbb{H};x_1,x_3,x_4)=\frac{1}{\sqrt{\pi}}\times \frac{(x_4-x_1)(x_3-x_1)}{\sqrt{x_4-x_3}}\times \frac{\left(\frac{1}{x_4-x_1}+\frac{1}{x_3-x_1}\right)(z-x_1)-2}{\sqrt{z-x_3}\sqrt{z-x_4}(z-x_1)},\quad z\in \mathbb{H},
\end{align}
and 
\begin{align}\label{eqn::def_obser_free_con_2}
	F(z;y_1,y_3,y_4):= \vert(\varphi^{-1})'(z) \vert^{\frac{1}{2}}\times  F(\varphi^{-1}(z);\mathbb{H};x_1,x_2,x_3),\quad z\in [-M,M]\times [0,M]\setminus \big(\{-M,M\}\times \{0,M\}\big). 
\end{align}
Note that the function $F$ is defined up to a sign. 
\begin{lemma}{\textnormal{\cite[Proposition~1.1]{IzyurovObservableFree}}}
	We have the following convergence of the scaled  observable (here $F^a$ is viewed as a function on midpoints of $\overline{\Omega}_M^a$):
	\begin{align*}
		2^{-\frac{1}{4}}a^{-\frac{1}{2}} F^a(\cdot) \to F(\cdot;y_1,y_3,y_4) \quad \text{locally uniformly as }a\to 0,
	\end{align*}
	where both sides are defined up to a sign and where $F(\cdot;y_1,y_3,y_4)$ is defined by~\eqref{eqn::def_obser_free_con_1} and~\eqref{eqn::def_obser_free_con_2}. 
\end{lemma}

\begin{lemma}
	We have the convergence
	\begin{align} \label{eqn::low_tempe_aux0}
	\lim_{a\to 0}	\vert 2^{-\frac{1}{4}} a^{-\frac{1}{2}} F^a(b_2^a) \vert= \vert F(y_2;y_1,y_3,y_4) \vert,
	\end{align}
where $F(\cdot;y_1,y_3,y_4)$ is defined by~\eqref{eqn::def_obser_free_con_1} and~\eqref{eqn::def_obser_free_con_2}. 
\end{lemma}
\begin{proof}
	The boundary of $\Omega_M^a$ near $y_2^a$ satisfies the regularity assumption in~\cite[Definition~3.14]{ChelkakIzyHolomorphic}. Thus, we can repeat the argument in~\cite[Proof of Lemma~4.8]{ChelkakIzyHolomorphic} to obtain the desired convergence on the boundary. 
\end{proof}
Now, we are ready to prove Lemma~\ref{lem::low_tempe}.
\begin{proof}[Proof of Lemma~\ref{lem::low_tempe}]
According to~\cite[Eq.~(1.7)]{IzyurovObservableFree}, we have 
\begin{align}\label{eqn::low_tempe_aux1}
	\vert F^a(b_2^a)\vert =\frac{1}{(\sqrt{2}-1)^{\frac{1}{2}}\cos\frac{\pi}{8}} \times \frac{Z(\overline{\Omega}_M^a;\{y_1^a,y_2^a\})}{Z(\overline{\Omega}_M^a;\{y_1^a,y_3^a-\ii a\})}.
\end{align}
Combining~\eqref{eqn::low_tempe_aux0} with~\eqref{eqn::low_tempe_aux1} gives~\eqref{eqn::low_tempe} and completes the proof. 
\end{proof}

\bigskip
\paragraph*{Declaration of competing interest.} The authors declare that they have no known competing financial interests or personal relationships that could have appeared
to influence the work reported in this paper.

\paragraph*{Acknowledgments.}
The authors thank Xin Sun and Baojun Wu for explaining their work~\cite{IntegraCLE}.
F. C.'s research is supported by NYU Abu Dhabi via a personal research grant. Y. F. thanks NYUAD for its hospitality during two visits in the fall of 2023 and of 2024.
The first visit was partially supported by the Short-Term Visiting Fund for Doctoral Students of Tsinghua University.

\end{document}